\documentclass[11pt,reqno]{amsart}
\usepackage{fancyhdr}
\usepackage{amssymb}
\usepackage{euscript}

\usepackage{amsmath,mathtools}
\usepackage{keytheorems}
\usepackage[usenames,dvipsnames]{color}
\usepackage[kerning=true]{microtype}
\usepackage[all,cmtip]{xy}
\usepackage{pgf,tikz}
\usepackage{ifthen}
\usepackage{wrapfig}
\usepackage{colortbl}
\usepackage{hhline}
\usetikzlibrary{decorations.markings,shadows}
\usetikzlibrary{arrows,shapes,positioning,calc,shadows}
\tikzstyle arrowstyle=[scale=1]
\tikzstyle directed=[postaction={decorate,decoration={markings,
    mark=at position .5 with {\arrow[arrowstyle]{stealth}}}}]
\usetikzlibrary{shapes.geometric}
\usepackage{fancybox}
\usetikzlibrary{shadows.blur}
\usepackage[most]{tcolorbox}
\usepackage{fullpage}
\usepackage{amsthm}
\usepackage{mathabx}
\usepackage{shuffle}
\usepackage{mathrsfs}

\usepackage{ytableau}
\ytableausetup{textmode, boxsize=.6cm}

\usepackage[font=small,figurename=Fig.]{caption}

\usetikzlibrary{backgrounds}
\usepackage{subcaption}

\usepackage[colorlinks=true, citecolor=cyan,urlcolor=blue,linktocpage=true,
	pagebackref,hyperindex=true,hypertexnames=true,hyperfigures=true]{hyperref} 

\definecolor{azure}{rgb}{0.,0.6,1.0} 
\definecolor{vertfonce}{rgb}{0.,0.7,0.} 
\definecolor{jaunefonce}{rgb}{1,.6,0} 
\newcommand\rouge{\textcolor{red}}
\newcommand\bleu{\textcolor{azure}}

\newtheorem{lemma}{\bleu{Lemma}}
\newtheorem{proposition}{\bleu{Proposition}}

\newtheorem{conjecture}{Conjecture}

\theoremstyle{remark}

\newcommand\autorefname[2]{\hyperref[#2]{#1~\ref*{#2}}}
\newcommand\autorefdef[1]{\hyperref[#1]{Definition~\ref*{#1}}}
\newcommand\autorefformula[1]{\hyperref[#1]{Formula~\ref*{#1}}}
\newcommand\autorefrec[1]{\hyperref[#1]{Recurrence~\ref*{#1}}}

\newcommand\unedefinition[1]{\bleu{\textbf{#1}}}
\newcommand{\newatop}[2]{\genfrac{}{}{0pt}{}{#1}{#2}}

\newcommand{\A}{\mathcal{A}}
\newcommand{\area}{\mathrm{area}}

\newcommand{\Anti}[1]{\widetilde{#1}}
\newcommand{\asc}{\mathrm{asc}}

\newcommand\cell{{\square}}
\newcommand\Cat{\mathcal{C}}
\newcommand\CatRec{\mathscr{C}}
\newcommand\CatNab{\EuScript{C}}

\newcommand\charac{\raise 2pt\hbox{\large$\chi$}}
\newcommand\scharac{\raise 2pt\hbox{$\scriptstyle \chi$}}

\newcommand{\des}{\mathrm{des}}
\newcommand{\DDesc}{\mathbb{D}}
\newcommand{\D}{\mathcal{D}}
\newcommand{\dinv}{\mathrm{dinv}}
\newcommand{\Der}{\mathrm{Der}}

\newcommand{\exc}{\mathrm{exc}}

\newcommand{\eint}[1]{[#1]}
\newcommand{\efact}[1]{[#1]!}
\newcommand{\ebinom}[2]{\genfrac{[}{]}{0pt}{}{#1}{#2}}

\newcommand\GL{\mathrm{GL}}

\newcommand\inv{\mathrm{inv}}

\newcommand\lcat{\widetilde{\mathcal{C}}}
\newcommand{\lint}[1]{\{#1\}}
\newcommand{\lfact}[1]{\{#1\}!}
\newcommand{\lbinom}[2]{\genfrac{\{}{\}}{0pt}{}{#1}{#2}} 
\newcommand{\lp}[1]{\Anti{p}_{#1}}
\newcommand{\lvarphi}[1]{\Anti{\varphi}_{#1}}

\newcommand{\maj}{\mathrm{maj}}
\newcommand{\meint}[1]{[#1]^{{\scriptscriptstyle (M)}}}
\newcommand{\mlint}[1]{\{#1\}^{{\scriptscriptstyle (M)}}}

\newcommand\N{\mathbb{N}}

\renewcommand\P{\mathbb{P}}

\newcommand\QQ{\mathbb{Q}}

\newcommand{\qtbinom}[2]{\genfrac{[}{]}{0pt}{}{#1}{#2}_{qt}}

\newcommand{\qtfact}[1]{[#1]{\lower4pt\hbox{$\stackrel{\textstyle !}{{}_{qt}}$\,}}}
\newcommand{\qtint}[1]{[#1]_{qt}}

\newcommand{\scalar}[2]{{\langle#1,#2 \rangle}}
\newcommand{\Sym}{\mathbb{S}}
\newcommand{\SYT}{{\textsc{syt}}}

\newcommand{\WCoxeter}{{\scriptscriptstyle W}}

\newcommand{\bx}{\boldsymbol{x}}

\newcommand\Z{\mathbb{Z}}

\numberwithin{equation}{section}
\numberwithin{theorem}{section}
\numberwithin{lemma}{section}
\numberwithin{proposition}{section}
\numberwithin{corollary}{section}
\numberwithin{definition}{section}
\numberwithin{conjecture}{section}
\numberwithin{figure}{section}

\title{A $(q,t)$-Overview of $q$-Analogs}
\author[F. Bergeron]{François Bergeron}
\address{Département de Mathématiques, Université du Québec à Montréal}
\urladdr{bergeron.math.uqam.ca}
\email{bergeron.francois@uqam.ca}
\thanks{This research was supported by a NSERC research grant}

\begin{document} 

\begin{abstract} By a simple homogenization process, one may turn a $q$-analog into a symmetric $(q,t)$-analog. Although simple, this bijective correspondence makes many concepts and identities become more natural, and proofs follow readily from classical properties of symmetric functions in two variables. A more intricate non-homogeneous extension is also developed. We illustrate this approach, revisiting recent developments in the study of $\gamma$-positivity, Lucas analogues, and the monoid of Cyclotomic generating functions. This also leads naturally to new constructions and conjectures. We further explore other avenues of investigations, included graded and equivariant $\gamma$-positivity and $\gamma$-anti-positivity (also known as alternatingly $\gamma$-positive).
 \end{abstract}

\maketitle

 \parskip=0pt
{ \setcounter{tocdepth}{1}\parskip=2pt\parindent=10pt\footnotesize \tableofcontents}
\parskip=6pt  
\parindent=20pt

\section{Introduction}
Just as for $q$-analogues, a polynomial $\alpha(q,t)\in \N[q,t]$ is a $(q,t)$-\unedefinition{analog} of an integer $n$, if $\alpha(1,1)=n$. We mostly consider here the ``simple'' case when $\alpha$ is a homogeneous polynomial, which is also symmetric in $q$ and $t$. For many classical $q$-analogs we may construct such a simple $(q,t)$-analog by simple homogenization. Thus, for any degree $d$ polynomial $a(q)\in\N[q]$ that is \unedefinition{palindromic} (or \unedefinition{self-reciprocal}), {\sl i.e.} such that $a(q) = q^d a(1/q)$, one obtains its $(q,t)$-version by setting
  $$\alpha(q,t) : = t^d a(q/t) \qquad \text{(homogenization)}.$$ 
As both $a(q)$ and its associated $a(q,t)$-analog carry the same information, there is nothing deep added by carrying out this homogenization. However, we will see that this brings clarity\footnote{Maybe reminiscent of  the effect of introducing homogeneous coordinates in projective geometry .} to many constructions involved in the study of $q$-analogues. Furthermore, this natural approach has been exploited in other contexts, see for instance in \cite{Ratliff2004} for Commutative Algebra or \cite{Fouvry2018} for Number Theory. We also briefly discuss the more subtle non-homogeneous case, which suggests many new directions of inquiry. 

When considered as $(q,t)$-symmetric polynomials, being unimodal directly corresponds to Schur-positivity, and $\gamma$-positivity corresponds to positivity in the elementary basis expansion. This natural approach can thus be of interest for many families of $\gamma$-positive polynomials that have recently been considered in a variety of papers \cite{AthanasiadisGamma,barrygammavectors,brittenham2016,Dey2020,Han2021,liao2026equivariant,ShareshianWachs}, as well as in the monograph \cite{Petersen2015}. Also tied to these lines of inquiry is the study of cyclotomic generating functions (see \cite{Billey2024,Gatzweiler}) which we link to the study of Lucas atoms (see \cite{Alecci2025,Bennett_2020,sagan2020lucas-17f}) via a natural isomorphism.

As several aspects of our approach consist of direct reformulations, it is naturally tied with previous work. We have tried to make this as explicit as possible with an extensive bibliography, referring to existing results that are essentially equivalent to those stated here. It is most probable that some references are missing. 


\section{Classical formulas and background} 
Typically, one starts with the $(q,t)$-\unedefinition{integers} $[n]_{qt}:=\frac{q^n-t^n}{q-t}=q^{n-1}+q^{n-2}t+\ldots+q\,t^{n-2}+t^{n-1}$ (with $[0]_{qt}:=0$), to subsequently introduce the $(q,t)$-\unedefinition{factorials}
 \begin{align}\label{qt_factorial}
   \qtfact{n}:= \begin{cases}
    [1]_{qt}[2]_{qt}\cdots [n]_{qt}  & \text{ if},\ n>0 \\[4pt]
      1 & \text{ if},\ n=0 .
\end{cases}
\end{align}
For instance:
\begin{align*}
&\qtfact{1}= 1,\\
&\qtfact{2}= q + t,\\
&\qtfact{3} = q^{3} + 2 q^{2} t + 2 q t^{2} + t^{3},\\
&\qtfact{4}= q^{6} + 3 q^{5} t + 5 q^{4} t^{2} + 6 q^{3} t^{3} + 5 q^{2} t^{4} + 3 q t^{5} + t^{6},\\
&\qtfact{5} = q^{10} + 4 q^{9} t + 9 q^{8} t^{2} + 15 q^{7} t^{3} + 20 q^{6} t^{4} + 22 q^{5} t^{5} + 20 q^{4} t^{6} + 15 q^{3} t^{7} + 9 q^{2} t^{8} + 4 q t^{9} + t^{10}.
\end{align*}
Next in line, one considers the $(q,t)$-\unedefinition{binomial coefficients} ({\sl aka}  $(q,p)$-binomials see \cite{Corcino}) directly defined as:
\begin{align}\label{qt_binomiaux}
 \qtbinom{n}{k} 
  	&:=\frac{\qtfact{n}}{\qtfact{k}\qtfact{n-k}},
 \end{align}
To check that these are indeed polynomials with positive integers, one simply verifies that they satisfy the  $(q,t)$-\unedefinition{Pascal triangle recurrence}
\begin{align}\label{q_Pascal_triangle}
     \qtbinom{n+k}{k} =q^k\qtbinom{n+k-1}{k}+t^{n} \qtbinom{n+k-1}{k-1},
\end{align}
with initial conditions $\qtbinom{n}{0} = \qtbinom{n}{n} =1$. From a combinatorial perspective, this $(q,t)$-Pascal recurrence may be understood in terms of the ``area-coarea'' enumeration of partitions contained in the $(n\times k)$ rectangle: 
\begin{align}\label{comb_qt_binomial}
   \qtbinom{n+k}{k}=\sum_{\alpha\subseteq (n^k)}q ^{|\alpha|} t^{nk-|\alpha|}.
\end{align}
As illustrated in \autorefname{Fig.}{Fig_Pascal}, the right-hand side of \autoref{comb_qt_binomial} is seen to satisfy \autoref{q_Pascal_triangle} by decomposing  the sum in two parts: with partitions having exactly $k$ parts on one hand, and those that have less than $k$ parts on the other. Some small values of $(q,t)$-binomial coefficients are given in \autoref{pascal_table}.\ytableausetup{aligntableaux=bottom,boxsize=.5cm}
\begin{small}
\begin{figure}
\begin{align*}
&
\hbox{\includegraphics[scale=.6]{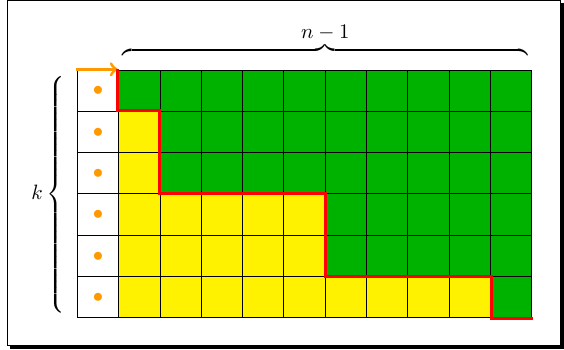}}
&&
{\includegraphics[scale=.6]{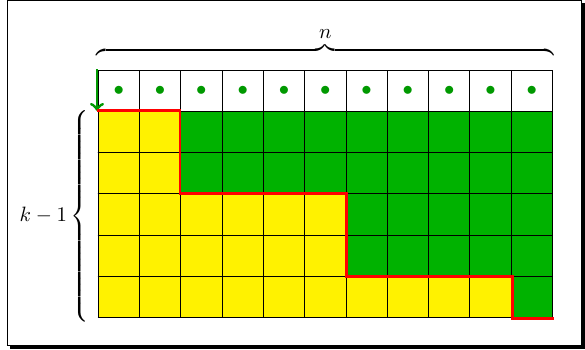}}\\
&\qquad\qquad \textcolor{orange}{q^k}\ \qtbinom{(n-1)+k}{k} &+&\qquad\qquad\textcolor{green!60!black}{t^n}\qtbinom{n+(k-1)}{(k-1)}
\end{align*}
\caption{The $(q,t)$-pascal triangle recurrence.}\label{Fig_Pascal}
\end{figure}
\end{small}


\begin{table}[ht]
  \centering 
\renewcommand{\arraystretch}{1.2}
\begin{small}
\begin{tabular}{|c||l|l|l|l|l|}
\hline
$\textstyle \qtbinom{n+k}{k}$ & $0$ & $1$ & $2$ & $3$\\[4pt] 
\hline
\hline
$0$ & $1$ & $1$ & $1$ & $1$ \\ 
\hline
$1$ & $1$ & $q + t$ & $\bleu{q^{2} + q t + t^{2}}$ 
		&  $q^{3} + q^{2} t + q t^{2} + t^{3}$  
           \\ 
\hline
$2$ & $1$ & $\rouge{q^{2} + q t + t^{2}}$  
		& $\begin{aligned} &q^2\bleu{(q^{2} + q t + t^{2})}\\ &\ +t^2\rouge{(q^{2} + q t + t^{2})}\end{aligned}$ 
		& $\begin{scriptstyle}\begin{aligned}\\[-2pt]   q^{6}&+ q^{5} t + 2 q^{4} t^{2} + 2 q^{3} t^{3}\\ 
											& + 2 q^{2} t^{4} + q t^{5} + t^{6}\\   \end{aligned}\end{scriptstyle}$ 
								       \\
\hline
$3$ & $1$ & $q^{3} + q^{2} t + q t^{2} + t^{3}$  
		& $\begin{scriptstyle}\begin{aligned}  q^{6}&+ q^{5} t + 2 q^{4} t^{2} + 2 q^{3} t^{3}\\ 
											& + 2 q^{2} t^{4} + q t^{5} + t^{6}   \end{aligned}\end{scriptstyle}$ 
		& $\begin{scriptstyle}\begin{aligned} \\ q^9&+ q^{8} t + 2 q^{7} t^{2}   \\ 
								&+ 3 q^{6} t^{3}+ 3 q^{5} t^{4} \\ 
								&+ 3 q^{4} t^{5} + 3 q^{3} t^{6} \\ 
								&+ 2 q^{2} t^{7} + q t^{8} + t^{9}\\   \end{aligned} \end{scriptstyle}$
								 \\
\hline
\end{tabular}
\end{small}
  \caption{$(q,t)$-Pascal triangle.}\label{pascal_table}
\end{table}
A $(q,t)$-version of the classical \unedefinition{Vandermonde identity}  goes as follows 
\begin{align}
   \qtbinom{a+b}{c} = \sum_{i+j=c} q^{j\,(a-i)} t^{i\,(b-j)}\qtbinom{a}{i}\qtbinom{b}{j}.
\end{align} 
Naturally, for $n=k_1+k_2+\ldots+k_j$, we may also consider the $(q,t)$-\unedefinition{multinomial coefficients}:
\begin{align}\label{q_multinomes}
 \qtbinom{n}{k_1,k_2,\ldots,k_j} 
  	&:=\frac{\qtfact{n}}{\qtfact{k_1}\qtfact{k_2}\cdots\qtfact{k_j}}.
 \end{align}
For upcoming developments, let us record the following equality (see \cite[Formulas 1 and 2]{Ma2024} and \cite[Page 8]{brittenham2016}): 
 \begin{align}
        \qtint{n+1} 
       			&= \sum_{0\leq 2k \leq n} \binom{n-k}{k} \rouge{(-qt)^k (q+t)^{n-2k}},\label{formule_qt_integer}
 \end{align}
and its clear corollary:
  \begin{align}
 	p_n(q,t)&=q^n+ t^n =\qtint{n+1}-qt\,\qtint{n-1}\nonumber\\
	             &=\sum_{1\leq 2k\leq n}\left( \binom{n-k}{k} + \binom{n-1-k}{k-1}\right) \rouge{(-qt)^k (q+t)^{n-2k}}, \label{qt_formule_power_sum}
 \end{align}
as they will play an interesting role in the sequel of our story. 
 
 \subsection{Further examples}
To better illustrate upcoming notions, let us also consider the following list of $(q,t)$-variant of classical formulas:
\begin{enumerate}
\item Partitions with distinct part sizes that are all less or equal to $n$:
\begin{align}\label{qt_Partition_distinct}
    \mathrm{Part}^\neq_{\leq n}(q,t) = \prod_{j=1}^n (q^j+t^j)
\end{align}
\item Alternating Sign Matrices:
\begin{align}\label{qt_ASM}
    \mathrm{ASM}_{n}(q,t) = \prod_{i=0}^{n-1}\frac{\eint{3i+1}_{qt}}{\eint{n+i}_{qt}}.
  \end{align}
\item Plane Partitions $\pi$ fitting inside a cube $a\times b\times c$ is:
\begin{align}\label{qt_PPP_abc}
      \mathrm{PP}_{abc}(q,t) := \sum_{\pi} q^{|\pi|} t^{abc-|\pi|}= 
      \prod_{i=1}^a\prod_{j=1}^b \prod_{k=1}^c \frac{\qtint{i+j+k-1}}{\qtint{i+j+k-2}}.
 \end{align} 
\item Totally Symmetric Plane Partitions fitting inside a cube $r\times r\times r$ :
\begin{align}\label{qt_TTSP}
    \mathrm{TSPP}_{r}(q,t) = \prod_{1\leq i\leq j\leq k\leq r}  \frac{\qtint{i+j+k-1}}{\qtint{i+j+k-2}}.
 \end{align}
 \item Standard Young Tableaux of shape $\mu$ (given by the \unedefinition{hook-length formula}):
\begin{align}
 \mathrm{SYT}_\mu(q,t)= \frac{\qtfact{n}}{\prod_{c\in \mu}{\qtint{\varepsilon_\mu(c)} }},
 \end{align}
 where $c\in\mu$ (also $\cell$) stands for $c$ being a cell of the partition $\mu$.
\item Semi-Standard Young Tableaux of shape $\mu$, with values in $\{1,2,\ldots, N\}$  (given by the \unedefinition{hook-content formula}):
\begin{align}\label{formula_SSYT}
  \mathrm{SSYT}^{{\scriptstyle N}}_\mu(q,t)=s_\mu(1,q,\ldots, q^{{\scriptstyle N}})= (qt)^{n(\mu)} \prod_{(i,j)\in \mu} \frac{\qtint{N+i-j}}{\qtint{\varepsilon_\mu(i,j)}},
 \end{align}
 with $n(\mu):=\sum_{(i,j)\in \mu} i$, and $\varepsilon_\mu(i,j)$ standing for the hook length of a cell $c=(i,j)$ of $\mu$.
  \end{enumerate}

\subsection{\texorpdfstring{$(q,t)$}{qt}-Catalan and \texorpdfstring{$(q,t)$}{qt}-Narayana} 

There are many different ``classical'' $q$-analogues or $(q,t)$-analogues of Catalan numbers, including the following: $ \Cat_n$, $\CatRec_{n}$ and $\CatNab_n$, described below. We first have the $(q,t)$-versions $ \Cat_n$ of the ``binomial formula'':
\begin{align}\label{def_qt_catalan_formula}
   \Cat_n(q,t):=\frac{1}{\qtint{n+1}} \qtbinom{2n}{n}.
  \end{align}
The second one, $\CatRec_{n}$, is obtained via the recurrence
 \begin{align}\label{def_qt_catalan_recurrence}
   \CatRec_{n+1}(q,t):=\sum_{k=0}^n q^k t^{n-k} \CatRec_{k}(q,t)\CatRec_{n-k}(q,t),
  \end{align}
 with initial condition $\CatRec_1=1$. The last is calculated via the Macdonald operator $\nabla$, or by enumeration of sub-partitions of the staircase partition $\delta_n = (n-1,\ldots,2,1)$ by $(\area,\dinv)$-statistics. It occurs as a bidegree enumerator of the alternating component of the module of diagonal harmonics  (see \cite{BergeronBook,haglund,haglund2015catalan} for more details). 
 \begin{align}
    \CatNab_n(q,t)&:=\scalar{ \nabla(e_n)}{e_n}\nonumber\\
    	&= \sum_{\alpha\subseteq \delta_n} q^{\binom{n}{2}-|\alpha|} t^{\dinv(\alpha)}. \label{def_qt_catalan_nabla}
   \end{align}
Here,  $e_n$ stands for the elementary symmetric function, and $\scalar{ -}{-}$ for the Hall scalar product (see \cite{Macdonald-Book-1995} for details). 
We have the specializations
  \begin{align}\label{qt_catalan_specializations}
       (qt)^{\binom{n}{2}}\, \CatNab_n\big(\textstyle \frac{q}{t},\frac{t}{q}\big) =  \Cat_n(q,t)\qquad \text{and}\qquad
       \CatNab_n(q,1) = \CatRec_{n+1}(q,1). 
  \end{align}
To compare the three, here are some small values:
  \begin{small}
  \begin{align*}
\Cat_{2}&=q^{2} + t^{2},
&& \CatRec_2=q + t,\\
&&& \CatNab_2  = q + t;
  \\[5pt]
\Cat_{3} & =q^{6} + q^{4} t^{2} + q^{3} t^{3} + q^{2} t^{4} + t^{6}, 
&&\CatRec_3  = (q^{3} + q^{2} t + q t^{2} + t^{3}) + q t, \\
&&& \CatNab_3  = (q^{3} + q^{2} t + q t^{2} + t^{3}) + q t;
\\[5pt]
\Cat_{4} & = q^{12} + q^{10} t^{2} + q^{9} t^{3} + 2 q^{8} t^{4} + q^{7} t^{5}
    &&\CatRec_4    = (q^{6} + q^{5} t + q^{4} t^{2} + 2 q^{3} t^{3} + q^{2} t^{4} + q t^{5} + t^{6}) \\
        &\ \ + 2 q^{6} t^{6} + q^{5} t^{7} + 2 q^{4} t^{8} + q^{3} t^{9} + q^{2} t^{10} + t^{12}, 
&&\qquad + (q^{4} t + q t^{4}) + (q^{3} t + 2 q^{2} t^{2} + q t^{3}), \\
    &&&\CatNab_4  =(q^{6} + q^{5} t + q^{4} t^{2} + q^{3} t^{3} + q^{2} t^{4} + q t^{5} + t^{6}) \\
    &&&\qquad  + (q^{4} t + q^{3} t^{2} + q^{2} t^{3} + q t^{4}) + (q^{3} t + q^{2} t^{2} + q t^{3}).
  \end{align*}
  \end{small}\noindent
  Notice that $\CatRec_n$ and $\CatNab_n$ are not homogeneous in general (see \autoref{graded_qt_analogs} below).  This makes their description more subtle, since one cannot recover them from a specialization at $t=1$, contrary to the homogeneous case. 
   
 Closely tied to Catalan numbers (see \cite[Formula (13)]{Ma2024} ) are  the \unedefinition{Narayana numbers} $N_{nk}$. They occur as coefficients of our next $(q,t)$-analog of Catalan number:
\begin{align}
\mathcal{N}_{n}(q,t)&:= \sum_{k=0}^{n} N_{nk}\,q^k t^{n-1-k} \nonumber\\
   &=\sum_{k=0}^{n} \frac{1}{{k+1}} \binom{n}{k}\binom{n-1}{k} q^k t^{n-1-k} \nonumber\\
   &=\sum_{0\leq 2j\leq n-1}  \Cat_{j+1}(1,1) \binom{n-1}{2j} (qt)^j (q+t)^{n-1-2j}.\label{qt_gamma_narayana}
\end{align}
We also have the bi-index notion of $(q,t)$-\unedefinition{Narayana}:
  \begin{align}\label{def_q_narayana_formula}
      N_{nk}(q,t) := \frac{1}{\qtint{k+1}} \qtbinom{n}{k}\qtbinom{n-1}{k}.
  \end{align}
These exhibit the symmetry $N_{(n,k)}(q,t) = N_{(n,n-k-1)}(q,t)$, so that we need only give values for  $k< n/2$. For instance, we have
\begin{table}[ht]
\renewcommand{\arraystretch}{1}
\begin{small}
\begin{tabular}{|c|l|l|l|l|l|l|}
\hline
$ N_{nk}(q,t)$ & $0$ & $1$ & $2$  \\[4pt] 
\hline
\hline
$1$ & $1$ & &\\
$2$ & $1$ & &\\
$3$ & $1$ & $q^{2} + q t + t^{2}$ & \\
$4$ & $1$ & $q^{4} + q^{3} t + 2 q^{2} t^{2} + q t^{3} + t^{4}$ & \\[-10pt]
$5$ & $1$ 
       & $\begin{aligned}\\[3pt]
       		q^{6} &+ q^{5} t + 2 q^{4} t^{2} + 2 q^{3} t^{3}\\
       		         & \qquad + 2 q^{2} t^{4} + q t^{5} + t^{6}
            \end{aligned}$ 
       & $\begin{aligned}\\[3pt]
       		q^{8} &+ q^{7} t + 3 q^{6} t^{2} + 3 q^{5} t^{3} + 4 q^{4} t^{4}\\
       		         &\qquad + 3 q^{3} t^{5} + 3 q^{2} t^{6} + q t^{7} + t^{8}
	    \end{aligned}$ \\[-20pt]
$6$ & $1$
       & $\begin{aligned}\\[3pt]
                 q^{8} &+ q^{7} t + 2 q^{6} t^{2} + 2 q^{5} t^{3} + 3 q^{4} t^{4}\\
       		         &\qquad  + 2 q^{3} t^{5} + 2 q^{2} t^{6} + q t^{7} + t^{8}
              \end{aligned}$
       & $\begin{aligned}\\[12pt]
        q^{12} &+ q^{11} t + 3 q^{10} t^{2} + 4 q^{9} t^{3} + 6 q^{8} t^{4}\\
       		         &\qquad   + 6 q^{7} t^{5} + 8 q^{6} t^{6} + 6 q^{5} t^{7} + 6 q^{4} t^{8}\\
	                 &\qquad + 4 q^{3} t^{9} + 3 q^{2} t^{10} + q t^{11} + t^{12}
	   	    \end{aligned}$ \\

\hline
\end{tabular}
\end{small}
  \caption{$(q,t)$-Narayana $N_{nk}$.}\label{qt_narayana_table}
\end{table}
We will come back to these examples to discuss their properties in upcoming sections.

\subsection{Factorial and Eulerian \texorpdfstring{$(q,t)$}{qt}-analogues}\label{sec_qt_eulerian}
As illustrated by some of the previous formulas, natural combinatorial $(q,t)$-analogues may be obtained via the $(q,t)$-enumeration of objects. For permutations,  classical instances of such statistics include (see \cite{MR2465404,Remmel-Wilson-2015} for more possibilities and variants):
\begin{itemize}
   \item \unedefinition{Exedance} number: $\exc(\sigma):=\# \{i \,|\, \sigma_i>i\}$, 
   \item \unedefinition{Descent} number: $ \des(\sigma):=\# \{i \,|\, \sigma_i>\sigma_{i+1}\}$, 
   \item \unedefinition{Ascent} number: $ \asc(\sigma):=\# \{i \,|\, \sigma_i<\sigma_{i+1}\}$,
    \item \unedefinition{Inversion} number: $ \inv(\sigma):=\# \{(i,j)\,|\, i<j \ \text{and}\ \sigma_i>\sigma_j\}$, and
   \item \unedefinition{Major index}: $\maj(\sigma):=\textstyle \sum_{\sigma_i>\sigma_{i-1}} i$.
\end{itemize}
These statistics are sometimes considered in conjunction with special families of permutations, and related structures. Among these, we have
\begin{itemize}
   \item The set of \unedefinition{derangements} $\Der_n:=\{\sigma \in \Sym_n\,|\, \sigma(i)\neq i\, \text{for all}\ i\}$,
   \item The set of \unedefinition{no double descent} permutations $\DDesc_n$, {\sl i.e.} no $i$ such that $\sigma(i)> \sigma(i+1)> \sigma(i+2)$). With $\DDesc'_n\subseteq \DDesc_n$ the subset of those that have no descent in last possible position (such that $\sigma(n-1)<\sigma(n)$); and $\DDesc''_n\subseteq \DDesc'_n$ for the subset of those for which we further have $\sigma(1)<\sigma(2)$.
   \item The set $\mathcal{R}_n$ of \unedefinition{compositions of $n$ having part size at most $2$}, which encodes descent ``runs'' for permutations lying in $\DDesc_n$. Similarly, the subset $\mathcal{R}'_n\subseteq \mathcal{R}_n$ of compositions that have last part equal to $1$; and $\mathcal{R}''_n\subseteq \mathcal{R}'_n$ for compositions having both first and last part equal to $1$. 
\end{itemize}

Both major-index and inversions give combinatorial interpretations for the ``classical'' $(q,t)$-factorial of \autorefdef{qt_factorial}:
 \begin{align*}
      \qtfact{n} & = \sum_{\sigma\in \Sym_n} q^{\inv(\sigma)}t^{\binom{n}{2}-\inv(\sigma)} \\
                     & =\sum_{\sigma\in \Sym_n} q^{\maj(\sigma)}t^{\binom{n}{2}-\maj(\sigma)}.
 \end{align*}
For a different tack on factorial $(q,t)$-analogues, we may also consider the \unedefinition{eulerian polynomials}, $A_n(q,t)$, with permutations enumerated by the descent statics (or excedances):
 \begin{align}
     A_{n}(q,t) :=\sum_{\sigma\in \Sym_{n}} q^{\des(\sigma)}t^{n-1-\des(\sigma)} \label{def_eulerian_pol}.
  \end{align}
Small values are as follows:
\begin{align*}
&A_1(q,t) = 1,\\
&A_2(q,t) = q + t,\\
&A_3(q,t) = q^{2} + 4 q t + t^{2},\\
&A_4(q,t) = q^{3} + 11 q^{2} t + 11 q t^{2} + t^{3},\\
&A_5(q,t) = q^{4} + 26 q^{3} t + 66 q^{2} t^{2} + 26 q t^{3} + t^{4},\\
&A_6(q,t) = q^{5} + 57 q^{4} t + 302 q^{3} t^{2} + 302 q^{2} t^{3} + 57 q t^{4} + t^{5}.
\end{align*} 
These also afford the explicit expression
 \begin{align}\label{formule_eulerian_pol} 
     A_n(q,t) = \sum_{k=0}^{n} \left(\sum_{j=0}^k (-1)^j \binom{n+1}{j} (k+1-j)^{n+1}\right) q^k t^{n-1-k}.
 \end{align}
 as well as formula\footnote{\cite[Thm 3.3]{MR4378084} gives a similar expression in terms of excedances.} (see \cite[Thm 5.6]{Foata})
\begin{align}\label{qt_Foata_Schutzenberger_formula}
   A_n(q,t) := \sum_{\sigma\in\DDesc'_n} (q+t)^{n-2\text{des}(\sigma)} (qt)^{\text{des}(\sigma)}, 
\end{align}
with $\sigma$ running over the set $\DDesc'_n$ of permutations $\sigma$ of $n$ that have no double descents (no $i$ such that $\sigma(i)> \sigma(i+1)> \sigma(i+2)$), such $\sigma(n-1)<\sigma(n)$.

The related \unedefinition{binomial eulerian polynomials}, $\A_n(q,t)$ are obtained from $A_n(q,1)$ by application of the \unedefinition{eulerian transformation}\footnote{See \cite{Athanasiadis2025}, \cite[Section 3.2]{MR4259159}, as well as \cite{MR4378084,Ji2024,Liu2025} for more on the {Eulerian transformation}. 
} and homogenization (see \cite[(1.1)]{Han2021}, \cite{postnikov2007}  and \cite{ShareshianWachs}), namely:
\begin{align}\label{qt_binomial_eulerian}
     {\A}_n(q,t):=\mathrm{Hom}_{qt}\left(\rouge{1+ q}\sum_{k=0}^n \binom{n}{k} A_k(q,1)\right).
\end{align}
We thus obtain
\begin{align*}
& \A_0 = 1,\\
& \A_1 = q + t,\\
& \A_2 = q^{2} + 3 q t + t^{2},\\
& \A_3 = q^{3} + 7 q^{2} t + 7 q t^{2} + t^{3},\\
& \A_4 = q^{4} + 15 q^{3} t + 33 q^{2} t^{2} + 15 q t^{3} + t^{4},\\
& \A_5 = q^{5} + 31 q^{4} t + 131 q^{3} t^{2} + 131 q^{2} t^{3} + 31 q t^{4} + t^{5},\\
& \A_6 = q^{6} + 63 q^{5} t + 473 q^{4} t^{2} + 883 q^{3} t^{3} + 473 q^{2} t^{4} + 63 q t^{5} + t^{6}.
\end{align*}
For which we have the formula\footnote{\cite[Thm 3.3]{MR4378084} gives a similar expression in terms of excedances.} 
\begin{align}\label{qt_binomial_eulerian_formula}
   \A_n(q,t):= 	\sum_{\rho\in \mathcal{R}_n} (q+t)^{|\rho|_{_1}} (qt)^{|\rho|_{_2}},
\end{align}
Among further interesting $(q,t)$-analogs associated to subsets of permutations, we have:
\begin{align}\label{qt_derangements}
	\D_n(q,t) 
				= \sum_{\sigma\in \rouge{\mathrm{Der}_n}} q^{\exc(\sigma)} t^{n-\exc(\sigma)}.
\end{align}
For which, when $n\geq 1$, we have $(q+t)^2\cdot\D_n(q,t) = (qt)\cdot A''_n(q,t)$, with 
\begin{align}\label{e_ribbons_pp}
	A''_n(q,t)=\sum_{\rho\in \mathcal{R}''_n} (q+t)^{|\rho|_{_1}} (qt)^{|\rho|_{_2}},
\end{align}
Here $|\rho|_{_i}$ denotes the number of size $i$ parts of a composition $\rho$ running over the set $\mathcal{R}''_n$ defined above.
Further examples are discussed in \autoref{sec_gamma_examples}.

%
%

 \subsection{Graded $(q,t)$-analogs}\label{graded_qt_analogs} 
 Most of the above examples are homogeneous and symmetric in $(q,t)$. We also have instances of non-homogenous $(q,t)$-symmetric enumerators such as the two of the Catalan numbers: the ones obtained via \autorefname{Recurrence}{def_qt_catalan_recurrence}; and those of \autorefformula{def_qt_catalan_nabla}, associated to the rich combinatorics of Macdonald polynomial theory. Still, these are symmetric in $(q,t)$. We will discuss below nice features of their \unedefinition{graded components}. In those graded instances, our study becomes more subtle. Indeed, there is no way to go back and forth between such a $(q,t)$-polynomial and its $t=1$ specialization, as was clearly possible in the homogeneous context. In other terms, the explicit grading of a $(q,t)$-analog is generally lost when passing to its $q$-analog. 
 
An instance occurs in  the enumeration of permutations with respect to the \unedefinition{major-comajor} weight: 
\begin{align}\label{major_comajor}
   \mathcal{M}_n(q,t):= \sum_{\sigma\in \Sym_n} q^{\maj(\sigma)}t^{\binom{n}{2}-\maj(\sigma^{-1})}
\end{align}
here illustrated with $n=4$:
  \begin{align*}
     \mathcal{M}_4(q,t)&=(q^{5} t^{3} + q^{4} t^{4} + q^{3} t^{5}) + (q^{5} t^{2} + q^{4} t^{3} + q^{3} t^{4} + q^{2} t^{5}) \\
     &\qquad + (q^{6} + q^{5} t + 2 q^{4} t^{2} + 2 q^{3} t^{3} + 2 q^{2} t^{4} + q t^{5} + t^{6}) \\
     &\qquad+ (q^{4} t + q^{3} t^{2} + q^{2} t^{3} + q t^{4})+ (q^{3} t + q^{2} t^{2} + q t^{3}),
 \end{align*}
 where we have grouped terms by ``total'' degree. Observe that, for $t=1$, we get
   \begin{align*}
     \mathcal{M}_4(q,1)&=(q^{5}  + q^{4}  + q^{3} ) + (q^{5}  + q^{4}  + q^{3} + q^{2}  ) \\
     &\qquad + (q^{6} + q^{5}  + 2 q^{4}  + 2 q^{3} + 2 q^{2}   + q   +1) \\
     &\qquad+ (q^{4}  + q^{3}   + q^{2}  + q  )+ (q^{3}  + q^{2}  + q)\\
     &= q^{6} + 3 q^5 + 5 q^4 + 6 q^3 + 5 q^3+ 3 q^2 +1,
 \end{align*}
 so that we evidently lose sight of the provenance of the various contributions to a given coefficient. 
Other typical well-known examples correspond to the  
Stirling polynomials of both kinds (see \cite{CaiReaddy}):
\begin{itemize}
\item The graded $(q,t)$-version of the sign-less Stirling numbers of the first kind (see \cite{LerouxMedicis}) may be obtained via the generating polynomial
\begin{align}\label{starling_first_kind}
	 \sum_{j=0}^{n-1} c_{nj}(q,t) \, z^j=
	  \prod_{k=1}^n (1+ \qtint{k}\cdot  z). 
\end{align}
Equivalently, we have the recurrence
   \begin{align}\label{rec_stirling_first_kind}
    	 c_{nk}(q,t):=c_{(n-1,k-1)}(q,t)+\qtint{n-1}\,c_{(n-1,k)}(q,t),
\end{align} 
with initials conditions $c_{n0}=c_{0k}=0$ and $c_{00}=1$.
For example, we get the graded symmetric polynomials of \autoref{stirling1_table}.
\begin{table}[ht]
\begin{small}
\begin{tabular}{|c||l|l|l|l|l|}
\hline
$ c_{nk}(q,t)$ & $0$ & $1$ & $2$ & $3$& $4$\\[4pt] 
\hline
\hline
$0$ & $1$ &  &  &   &   \\
$1$ & $0$ & $1$ &   &   &  \\
$2$ & $0$ & $1$ & $1$ &   &   \\
$3$ & $0$ & $q + t$ & $(q + t) + 1$ & $1$ &   \\[-14pt] 
$4$ & $0$ & $q^{3} + 2 q^{2} t + 2 q t^{2} + t^{3}$ 
	& $\begin{aligned}\\[7pt] (q^{3} &+ 2 q^{2} t + 2 q t^{2} + t^{3})\\
					      &+ (q^{2} + q t + t^{2}) + (q + t)\\[-4pt] \end{aligned}$ 
	& $\begin{aligned}\\[-4pt](q^{2} &+ q t + t^{2})\\
					      & + (q + t) + 1 \end{aligned}$ & $1$\\
\hline
\end{tabular}
\end{small}
  \caption{Sign-less $(q,t)$-Stirling of the first kind.}\label{stirling1_table}
\end{table}

\item  In turn, the $(q,t)$-version of the Stirling polynomials of the second kind (see \cite{Wachs1991}) arises via the recurrence
\begin{align}\label{qt_Stirling_second_kind}
  S_{nk}(q,t):=S_{(n-1,k-1)}(q,t)+\qtint{k}\, S_{(n-1,k)}(q,t),
 \end{align}
 as displayed in \autoref{stirling2_table}.
 \begin{table}[ht]
\begin{small}
\begin{tabular}{|c|l|l|l|l|l|l|}
\hline
$ S_{nk}(q,t)$ & $0$ & $1$ & $2$ & $3$& $4$ \\[4pt] 
\hline
\hline
$0$ &$1$ &   &   & &    \\
$1$ &$0$ & $1$ &   &  &     \\
$2$ &$0$ & $1$ & $1$ &   &    \\
$3$ &$0$ & $1$ & $q + t + 1$ & $1$ &     \\
$4$ &$0$ & $1$ & $(q^{2} + 2 q t + t^{2}) + (q + t) + 1$ & $(q^{2} + q t + t^{2}) + (q + t) + 1$ & $1$   \\
\hline
\end{tabular}
\end{small}
  \caption{$(q,t)$-Stirling of the second kind.}\label{stirling2_table}
\end{table}
\end{itemize}
We thus decompose positive integer coefficient polynomials, in $\P=\N[q,t]$, into their various homogeneous components lying in
$\P_d := \N\{ q^a t^b\,|\, a+b=d\}$. This clearly gives a grading of the semi-ring 
  $\P= \bigoplus_{d\geq 0} \P_d$.
As discussed in the next section, we are interested in symmetric polynomials component of this semi-ring.

 \section{Symmetric polynomials in two variables}\label{direct_qt_analog}
 Many notions appearing in the study of $q$-analogues become ``natural'' when the polynomials involved are homogenized so that they may be considered as symmetric functions in two variables $q$ and $t$. Just to illustrate, being a \unedefinition{palindromic} polynomial is just saying that the homogenized version is symmetric, and the two bases  considered in \cite[Section 2.]{brittenham2016} simply correspond to the (two variables) Schur and elementary symmetric functions bases. We will also extend such questions to graded settings.
 
 For the \unedefinition{unimodality} property of a $(q,t)$-symmetric polynomial, we easily verify the following:
    \begin{tcolorbox}[colback=azure!5!white]
 \begin{lemma}
   A  symmetric $(q,t)$-polynomial in $\N[q,t]$ is unimodal by degree, if and only if its Schur expansion has positive\footnote{For those who wonder, this means that coefficints lie in $\N=\{0,1,2,\ldots\}$, weirdly called non-negative by some.} coefficients.
 \end{lemma}
 \end{tcolorbox}
As we will see below, many interesting $(q,t)$-analogues are unimodal, hence Schur-positive.  As an illustration, the unimodality of
 $$\qtfact{5} = q^{10} + 4 q^{9} t + 9 q^{8} t^{2} + 15 q^{7} t^{3} + 20 q^{6} t^{4} + 22 q^{5} t^{5} + 20 q^{4} t^{6} + 15 q^{3} t^{7} + 9 q^{2} t^{8} + 4 q t^{9} + t^{10},$$
 corresponds to the sum
 \begin{small}
 \begin{gather*} 
2\cdot(q^{5} t^{5})\\
+\,5\cdot(q^{6} t^{4} + q^{5} t^{5} + q^{4} t^{6})\\
+\,6\cdot(q^{7} t^{3} + q^{6} t^{4} + q^{5} t^{5} + q^{4} t^{6} + q^{3} t^{7})\\
+\,5\cdot(q^{8} t^{2} + q^{7} t^{3} + q^{6} t^{4} + q^{5} t^{5} + q^{4} t^{6} + q^{3} t^{7} + q^{2} t^{8})\\
+\,3 \cdot(q^{9} t + q^{8} t^{2} + q^{7} t^{3} + q^{6} t^{4} + q^{5} t^{5} + q^{4} t^{6} + q^{3} t^{7} + q^{2} t^{8} + q t^{9})\\
+(q^{10} + q^{9} t + q^{8} t^{2} + q^{7} t^{3} + q^{6} t^{4} + q^{5} t^{5} + q^{4} t^{6} + q^{3} t^{7} + q^{2} t^{8} + q t^{9} + t^{10}).
\end{gather*}
\end{small}\noindent
Among other notions that become clearer in this setting are the notions of ``$\gamma$-positivity'' and ``$\gamma$-anti-positivity (see \autoref{sec_gamma_positivity}), as well as related properties.  Here are some of the expansions of the diagonal harmonics $(q,t)$-Catalan $\CatNab_n$ (see \autorefname{definition}{def_qt_catalan_nabla}) in the Schur-basis, graded by degree:
 \begin{align*}
\CatNab_2  &=s_{1},\\
\CatNab_3  &=s_3+s_{11} ,\\
\CatNab_4  &=s_6+s_{41} +s_{31}, \\
\CatNab_5  &=s_{(10)}+ s_{81} +(s_{71}+ s_{62}) + (s_{61} + s_{43}) +s_{42},\\ 
\CatNab_6  &=s_{(15)}+ s_{(13,1)} + (s_{(12,1)}+ s_{(11,2)}) +(s_{(11,1)} + s_{(10,2)}+ s_{93} ) \\
           &\qquad + (s_{(10,1)}+  s_{92} + s_{83}+ s_{74})+(s_{82}   +s_{73}+ s_{64}) + (s_{72}  + s_{63}) + s_{44}.
\end{align*}
In these, each homogeneous degree-component is  Schur positive, and thus $(q,t)$-unimodal. We will say that such a $(q,t)$-polynomial is \unedefinition{unimodal by degree}.

\subsection{Classical symmetric function background}
We recall here well-known properties (see the classical reference \cite{Macdonald-Book-1995}) and identities concerning the graded subring  $\Lambda_2=\bigoplus_{d\geq 0}\Lambda_2^{(d)}$ of the polynomial ring $\QQ[q,t]$, whose elements are symmetric. Its degree-$d$ homogeneous component $\Lambda_2^{(d)}$ affords as a basis the set of \unedefinition{monomial symmetric} polynomials indexed by partitions with at most two parts: $\{m_{ab}\}_{ab}$ where $a\geq b\geq 0$ and $a+b=d$. 
Recall that these are simply
   $$m_{ab} = q^at^b+q^bt^a.$$
Special instances include \unedefinition{power-sums} $p_n = m_{n} = q^n+t^n$, and elementary symmetric functions $e_1= m_1 = q+t$ and $e_2= m_{11} = qt$. 
We also have the Schur basis $\{s_{ab}\}_{ab}$, with $a\geq b\geq0$ and $a+b=d$ (all other Schur polynomials vanish):
   $$s_{ab} = q^at^b + q^{a-1} t^{b+1} + \ldots + q^{b+1} t^{a-1} + q^b t^a $$
The corresponding transition matrix is the restriction of the Kostka matrix to partitions having at most two parts:
 \begin{align}\label{two_kostka_matrix}
     K_{ab,cd} =\begin{cases}
     1 & \text{if}\ a\geq c, \\
      0 & \text{otherwise},
\end{cases}
\end{align}
or equivalently
	$$s_{ab} = \sum_{0\leq 2i\leq a-b}  m_{(a-i,b+i)}.$$ 
Of particular interest for the following is the basis\footnote{It may be worth recalling that the number of partitions in at most two parts is the same as the number of partitions with parts at most of size $2$.}  $\{e_1^j e_2^k \}_{j+2k=d}$. Observe that all other elementary $e_\mu$ vanish since $e_k(q,t)=0$ for all $k\geq 3$. Furthermore, the set $\{h_1^j h_2^k\}_{j+2k=d}$ forms yet another basis, as well as the set of power-sums $\{p_1^j p_2^k\}_{j+2k=d}$. Observe that all of the above basis polynomials actually lie in graded components of the semi-ring $\P$ (have positive coefficients in their $(q,t)$-expansion, and are homogenous).

 \subsection{Elementary basis expansion} Another useful expansion for symmetric $(q,t)$-polynomials in $\N[q,t]$ is their expansion as $\mathbb{Z}$-linear combinations of terms of the form $(q+t)^i (qt)^j$. In other terms, this is their expansion in the elementary symmetric function basis $\{e_\mu\}_{\mu}$, where $\mu$ runs over the set of partitions having parts smaller or equal to $2$ (as they involve only two variables). For sure, the special case when all coefficients are positive is of particular interest. We will come back later to the situation when  we have this positivity up to replacing $e_2$ by $(-e_2)$. 
 
 For the $(q,t)$-integers, we have the following:
    \begin{tcolorbox}[colback=azure!5!white]
\begin{lemma}\label{lemma_e_int}
	The $(q,t)$-integers $[n+1]$ expand as
    	\begin{align}\label{e_int_anti_positive}
              [n+1] = \sum_{2j\leq n} \binom{n-j}{j}e_1^{n-2j}(-e_2)^j.
   	\end{align}   
Equivalently, one has the recurrence $ [n+1] =e_1\,[n]-e_2\,[n-1]$, with initial conditions $[0]=0$ and $[1]=1$, or in matrix terms 
\begin{align}\label{matrix_formula_int}
          \begin{pmatrix} \eint{n+1}& (-e_2)\cdot\eint{n} \\
     			   \eint{n} & (-e_2)\cdot\eint{n-1}\end{pmatrix}
	=
	\begin{pmatrix} e_1& -e_2 \\ 
     			    1 & 0\end{pmatrix}^n.
\end{align}
\end{lemma}
\end{tcolorbox}
As $h_{n}(q,t)=\qtint{n+1}$ this follows from the dual Jacobi-Trudi identity\footnote{Not forgetting that $0=e_3=e_4=\ldots$, as we have only two variables.} modulo (see  \autorefformula{formule_qt_integer}):
\begin{align*}
      h_n=[n+1]= \det\begin{pmatrix}
e_1 & e_2 & 0 & \ldots & 0 & 0\\
1 & e_1 & e_2 & \ldots & 0 & 0\\
0 & 1 & e_1 & \ldots & 0 & 0\\
\vdots & \vdots & \vdots & \ddots & \vdots & \vdots\\
 0  &  0  & 0 &  \ldots  &  e_1 & e_2 \\
0  &  0  & 0 &  \ldots  &  1 & e_1
\end{pmatrix}
\end{align*}
with the matrix being of size $n\times n$.  As a corollary, when $a>b$, we may also express the Schur functions as
    	\begin{align}\label{s_to_e} 
              s_{ab} =\eint{a-b+1} \cdot e_2^b  = \sum_{2j\leq a-b} (-1)^j \binom{a-b-j}{j}e_1^{a-b-2j}e_2^{b+j}.
   	\end{align}   
or in matrix terms:
    	\begin{align}\label{s_to_e_mat} 
             \begin{pmatrix} s_{(a+1,b)}& -e_2\cdot s_{ab}\\ 
     			    s_{ab} & -e_2\cdot  s_{(a-1,b)}\end{pmatrix}
	      =
	    e_2^b\cdot \begin{pmatrix} e_1& -e_2 \\ 
     			    1 & 0\end{pmatrix}^{a-b+1}. 
   	\end{align}   
Furthermore, as $p_n= [n+1]-e_2\,[n-1]$, we get the $e$-expansion of the power-sum (see \autorefformula{qt_formule_power_sum}):
\begin{align}
 	 p_n = \sum_{0\leq 2j \leq n} \frac{n}{n-j} \binom{n-j}{j} e_1^{n-2j}(-e_2)^{j}.\label{p_to_e}
\end{align}
Equivalently, we have the recurrence
\begin{align}\label{rec_e_power}
    p_n =e_1\,p_{n-1} -e_2\,p_{n-2},
\end{align}
with initial conditions $p_1=e_1$ and $p_2=e_1^2-2\,e_2$.
Also coined in terms of a well-known generating series formula  (see \cite{Macdonald-Book-1995}), we have
 \begin{align}\label{GF_e_power}
    \sum_{n\geq 1} p_n x^n  = \frac{e_1\, x+2e_2\,x^2}{1-e_1\,x+e_2\,x^2}.
\end{align}
For the monomial basis, one has the $e$-expansion
\begin{align}\label{monomial_to_e}
 	 m_{ab} = \sum_{0\leq 2k \leq a-b} (-1)^k \frac{a-b}{a-b-k} \binom{a-b-k}{k} e_1^{a-b-2k}e_2^{k+b}.
\end{align}
The effect of plethysms $p_n\circ (-)$ on elements of $\Lambda_2$, are obtained via the calculation  rules:
\begin{align*}
      &p_n\circ e_1= [n+1]-e_2\,[n-1], &&p_n\circ e_2= e_2^n,\\
     &p_n\circ(a\,f+b\,g) =  a\,(p_n\circ f)+b\,(p_n\circ g)  &&p_n\circ(f\cdot g) =  (p_n\circ f)\cdot (p_n\circ g), 
\end{align*}
for any $f$ and $g$ in $\Lambda_2$, and scalars $a$ and $b$. More generally,  we may extend this to any $h\in \Lambda$, considering $h\circ(-)$ as an action of elements  on $\Lambda_2$, exploiting the further rules 
\begin{align*}
	&(h_1+h_2)\circ (-) =  (h_1\circ (-))+h_2(\circ (-)), &&\text{and}&& (h_1\cdot h_2)\circ (-) =  (h_1\circ (-))\cdot h_2(\circ (-)).
\end{align*} 
In particular, we get (see \cite[Prop. 4.6]{Ratliff2004}) 
\begin{align}\label{formule_produit_qt_int}
    (p_k\circ [n]) =\frac{ [nk]}{ [k]} .
\end{align}



\subsection{Back to \texorpdfstring{$(q,t)$}{qt}-analogues}
The recurrence for the $(q,t)$-{factorial} may be formally stated in $\Lambda_2$ as
 \begin{align}
   \efact{n+1}=s_n \cdot \efact{n},
\end{align}
with $\efact{0}= \efact{1}=1$. Notice that we typically omit writing variables when dealing with symmetric functions. We may be readily expanded the above recursively using $2$-variable version of the classical \unedefinition{Pieri formula} multiplication rule for Schur functions: 
\begin{align}\label{Pieri-Formula}
     s_n\cdot s_{ab} = \sum_{\newatop{k+j=n}{a+b\geq b+j}} s_{(a+k,b+j)}.
\end{align}
 Small values of the Schur expansion of the $(q,t)$-factorials are thus as follows:
\begin{align*}
 \efact{1}=&s_{0},\\
\efact{2} =&s_{1},\\
\efact{3}=&s_{3} + s_{21},\\
\efact{4}=&s_{6} + 2 s_{51} + 2 s_{42} + s_{33}, \\
\efact{5} =&s_{(10)} + 3 s_{91} + 5 s_{82} + 6 s_{73} + 5 s_{64} + 2 s_{55},\\
\efact{6} =&s_{(15)} + 4 s_{(14,1)} + 9 s_{(13,2)} + 15 s_{(12,3)} + 20 s_{(11,4)} + 22 s_{(10,5)} + 19 s_{96} + 11 s_{87}.
\end{align*}
 Observe that, in lexicographic, order, coefficients of the Schur expansion of $\efact{n} $ are themselves unimodal  (See \cite{MurtyUnimodal}). One may check that they are in fact log-concave. 
 
The $(q,t)$-\unedefinition{binomial coefficients}  correspond to the $(q,t)$-evaluation  of the plethysm $h_n\circ h_k$ :
\begin{align}
  \qtbinom{n+k}{k} &=\sum_{\alpha\subseteq (n^k)}q ^{|\alpha|}t ^{nk-|\alpha|}\nonumber\\
          &=(h_n\circ h_k)(q,t).\label{plethysm_qt_binomial}\\
\end{align}
With $i+j=b$, let us picture the terms $q^it^j$ of $h_b(q,t)=q^b+q^{b-1}t+\ldots + q\,t^{b-1}+t^b$ as a bicoloured row of cells:
\ytableausetup{mathmode, boxsize=.4cm}
\begin{align*}
&\ydiagram[*(yellow)] {6} *[*(green!70!black)]{11}\\[-9pt] 
 &\hskip-.1cm\underbrace{\hskip2.45cm}_{\textstyle i} \underbrace{\hskip2cm}_{\textstyle j}
\end{align*}
As above, the various monomials occurring in $(h_a\circ h_b)(q,t)$ are obtained in terms of tableaux-of-tableaux of the form (here with $a=4$ and $b=11$):
\ytableausetup{mathmode, boxsize=.3cm}
 $$  \begin{tikzpicture}[scale=.75]
  \draw[fill =azure!20] (-2.5,0.5) rectangle (2.5,1.5);  \node at (0,1) {\ydiagram[*(yellow)] {10} *[*(green!70!black)]{11}};
  \draw[fill =azure!20] (2.5,0.5) rectangle (7.5,1.5);  \node at (5.02,1)  {\ydiagram[*(yellow)] {6} *[*(green!70!black)]{11}};
  \draw[fill =azure!20] (7.5,0.5) rectangle (12.5,1.5); \node at (10.02,1) {\ydiagram[*(yellow)] {6} *[*(green!70!black)]{11}};
  \draw[fill =azure!20] (12.5,0.5) rectangle (17.5,1.5); \node at (15.02,1) {\ydiagram[*(yellow)] {2} *[*(green!70!black)]{11}};
\end{tikzpicture}$$
The associated weight being $q^k t^l$, if there are $k$ yellow cells in total, and $l$ green ones. This may be bijectively encoded by stacking the internal coloured row tableaux, to obtain \autorefname{Fig.}{fig_taxi}.

\begin{figure}[ht]
\includegraphics[scale=.7]{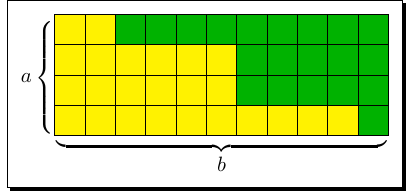}
\caption{Subpartion of a rectangle.}
\label{fig_taxi}
\end{figure}
Thus, we recuperate the description of \autoref{plethysm_qt_binomial} for the $(q,t)$-binomial coefficient, as a weighted enumerator of partitions contained in the $(a\times b)$-rectangle.
It follows from a representation theory argument  that this is a Schur positive polynomial. Indeed, the plethysm $(h_n\circ h_k)$ corresponds to $\GL(V)$-character of the symmetric power of symmetric power $S^{n}(S^{k}(V))$. Assuming that $\dim(V)=2$, this character expands as a positive integer linear combination of Schur functions indexed by partitions with at most two parts. For general $n$ and $k$, the coefficients of this Schur expansion in lex-order are not unimodal in general (see an example below). In view of the equivalent $(q,t)$-Pascal triangle recurrences (see \autoref{q_Pascal_triangle}): 
\begin{align*}
     \qtbinom{n+k}{k} &=q^k\qtbinom{n+k-1}{k}+t^{n} \qtbinom{n+k-1}{n}\\ 
		&=t^k\qtbinom{n+k-1}{k}+q^{n} \qtbinom{n+k-1}{n} 
\end{align*}
we get
\begin{align}\label{e_recurrence_binom}
\ebinom{n+k}{k}=\frac{p_k}{2} \, \ebinom{n+k-1}{k}+\frac{p_{n}}{2}\, \ebinom{n+k-1}{n}
\end{align}
As explained above, we also have
 \begin{align}\label{e_plethysm_binom}
\ebinom{n+k}{k}=h_n\circ \eint{k+1},
\end{align}
with plethysm considered as an action of $\Lambda$ on $\Lambda_2$ (see \autoref{sec_classical_prop_cyclo}, for more on this).
Instances of Schur expansions occurring in the $(q,t)$-Pascal triangle are given in \autoref{tab_schur_positivity_binomials}.
\begin{table}[ht]
\begin{small}
\begin{tabular}{|c||l|l|l|l|l|}
\hline
$\textstyle \ebinom{n+k}{k}$ & $0$ & $1$ & $2$ & $3$ & $4$  \\[4pt]
\hline
\hline
$0$ & $1$ & $1$ & $1$ & $1$ & $1$ \\
\hline
$1$ & $1$ & $s_{1}$ & $s_{2}$ & $s_{3}$ & $s_{4}$  \\
\hline
$2$ & $1$ & $s_{2}$ & $s_{4}+s_{22}$ & $ s_{6}+ s_{42}$ & $s_{8} + s_{62} + s_{44} $  \\
\hline
$3$ & $1$ & $s_{3}$ & $s_{6}+ s_{42}$ & $s_{9} + s_{72} + s_{63}$ & $ s_{(12)} + s_{(10,2)} + s_{93} + s_{84} + s_{66} $  \\[2pt]
\hline
$4$ & $1$ & $s_{4}$ & $s_{8} + s_{62} + s_{44} $ & $s_{(12)} + s_{(10,2)} + s_{93} + s_{84} + s_{66} $ & $\begin{aligned}  s_{(16)} &+ s_{(14,2)} + s_{(13,3)} \\ & + 2 s_{(12,4)}  + 2 s_{(10,6)} + s_{88}  \end{aligned}$ \\
\hline
\end{tabular}
\end{small}
\caption{Schur positivity of $(q,t)$-binomials}\label{tab_schur_positivity_binomials}
 \end{table}

 \subsection{The \texorpdfstring{$(ad\--bc)$}{adbc}-conjecture}
The $(q,t)$-specialization\footnote{A simple restriction to two variables.} of a conjecture stated in \cite{VESSENES2004579} gives the following (see  \cite{BergeronFoulkes,Zanello2018}). 
   \begin{tcolorbox}[colback=azure!5!white]
  \begin{conjecture}[$(ad\--bc)$-Conjecture]\label{conj_bergeron_vessenes}
    For all integers $a\leq b\leq c\leq d$ such that $ad=bc$, the polynomial
    \begin{align}\label{bergeron_vessenes}
         \ebinom{b+c}{b} -\ebinom{a+d}{a}
    \end{align} 
is Schur-positive (hence $(q,t)$-unimodal).
  \end{conjecture}
    \end{tcolorbox}
  \noindent
  Some cases have been shown to hold (see \cite{Zanello2018}). In many instances, but not all, the difference considered is $\gamma$-anti-positive (up to a global sign). The only exceptions for $n=ad=bc \leq 40$ correspond the following instances of $((a,d),(b,c))$
 \begin{align*}
&((2, 8),(4, 4)),&&((2, 12),(4, 6)),&&((2, 14),(4, 7)),&&((2, 15),(5, 6)),&&((2, 16),(4, 8)),\\
&((2, 18),(4, 9)),&&((2, 18),(6, 6)),&&((4, 9),(6, 6)),&&((2, 20),(5, 8)),&&((2, 20),(4, 10)).
 \end{align*}
Further experiments (with $n$ going up to $300$) suggest that we do have $\gamma$-anti-positivity when $a$ is odd.  It would be interesting to characterize all the cases for which we do have $\gamma$-anti-positivity. 
Observe that one may greatly reduce the number of instances to be verified for the above statements by exploiting the evident equality (as polynomials in $\N[q,t]$ or their Schur expansions):
  $$ \ebinom{b+c}{b} -\ebinom{a+d}{a}= \left(\ebinom{b+c}{b} -\ebinom{x+y}{x}\right)+ \left(\ebinom{x+y}{x} -\ebinom{a+d}{a}\right),$$
for values of $x$ and $y$ that satisfy the necessary conditions.

%

\section{\texorpdfstring{$\gamma$}{gamma}-Positivity and \texorpdfstring{$\gamma$}{gamma}-anti-positivity}\label{sec_gamma_positivity}

In classical settings (see  \cite{Petersen2015}), one considers ``$\gamma$-positivity'' for palindromic polynomials $f(q)$ in $\N[q]$. Recall that palindromicity means that, for some $n$, one has $f(q)=q^n\,f(1/q)$. When this holds, $f(q)$ affords a \unedefinition{$\gamma$-expansion}
  $$f(q)= \sum_{2i+j=n} \gamma_{ij}\,{q}^i (1+q)^j.$$
 If the coefficients $\gamma_{ij}$ are positive integers, one says that $f$ is \unedefinition{$\gamma$-positive}. For such $\gamma$-expansions, the following ``strange'' transformation is considered in the literature (see \cite{brittenham2016})
             $$f(q)= \sum_{2i+j=n} \gamma_{ij}\,\rouge{q}^i (1+q)^j \longmapsto  \sum_{2i+j=n} \gamma_{ij} \rouge{(-q)}^i (1+q)^j = \Anti{f}(q),$$
which ``selects'' some of the $q$'s to change their signs. 
This transformation appears much more natural in the $(q,t)$-symmetric polynomial setting, with properties derived easily.        
Indeed, in the $(q,t)$-homogenized setting, the property of being palindromic simply corresponds to a $(q,t)$-polynomial being symmetric; and the $\gamma$-expansion is the usual expression of the polynomial in the elementary basis $\{e_1^i e_2^j\}_{i+2j=n}$ (going back to Newton). Thus described, the $\gamma$-positivity of a polynomial is just $e$-positivity (and vice versa). Similarly, \unedefinition{$\gamma$-anti-positivity} corresponds to $e$-positivity up to the $\Anti{(-)}$ isomorphism:
\begin{align}
	\Anti{(-)}:\Anti{\Gamma} \longrightarrow \Gamma,\qquad \text{sending}\qquad \alpha(e_1,e_2)\mapsto \alpha(e_1,-e_2).
\end{align}
 In the study of polytopes,  $\gamma$-positive and $\gamma$-anti-positive\footnote{These also appear as under the guise of ``alternating gamma vectors'', see \cite{brittenham2016}.} polynomials are linked to face enumerations \cite{AthanasiadisGamma,Athanasiadis2025,Dey2020,Gal2005}. 

\subsection{The semimodules \texorpdfstring{$\Anti{\Gamma}$}{gam} and \texorpdfstring{$\Gamma$}{gamp}} 
A polynomial $f(q,t)\in \Z[e_1,e_2]$ is said to be \unedefinition{$\gamma$-positive} if it lies in the graded semimodule of polynomials: $\Anti{\Gamma}:=\N[e_1,e_2]$, with:
\begin{align}
   \Anti{\Gamma} = \bigoplus \Anti{\Gamma}_{d},\qquad \text{with}\qquad \Anti{\Gamma}_{d}:=\Z\{e_1^ie_2^j\, |\, i+2j=d\},
\end{align}
with $\deg(e_1)=1$ and $\deg(e_2)=2$, as usual. Although this is simply the classical notion of \unedefinition{$e$-positivity} in the context of the study of symmetric functions, we use the above terminology to better relate our discussion to existing extensive work in the context of $q$-analogues (see \cite{AthanasiadisGamma}). Clearly, $\gamma$-positivity implies Schur-positivity.

For reasons that will be discussed further below, it is also interesting to also consider polynomials $g(e_1,e_2) \in \mathbb{Z}[e_1,e_2]$, such that $g(e_1,-e_2)$ is $\gamma$-positive. We then say that they are \unedefinition{$\gamma$-anti-positive}. Thus we get the graded semimodule $\Gamma:=\N[e_1,-e_2]$ of  \unedefinition{$\gamma$-anti-positive polynomials}. 
We clearly have
   \begin{tcolorbox}[colback=azure!5!white]
\begin{lemma}\label{lem_gamma_anti}
The map $\Anti{(-)}$ is an involutive isomorphism of graded semi-modules.
\end{lemma}
\end{tcolorbox}

\subsection{Some $\gamma$-positive examples}\label{sec_gamma_examples}
We may recast \autorefformula{qt_gamma_narayana} as:
\begin{align}
\mathcal{N}_{n}=\sum_{0\leq 2j\leq n-1}  \Cat_{j+1}(1,1) \binom{n-1}{2j} e_2^j e_1^{n-1-2j},\label{e_gamma_narayana}
\end{align}
hence $\mathcal{N}_{n}$ is $\gamma$-positive.  The \unedefinition{exedance}-enumeration of derangements (see  \autoref{qt_derangements}), gives another instance of $\gamma$-positivity, with small examples as follows. 
 \begin{align*}
\D_2 &= s_{11} &&= e_{2},\\
\D_3 &= s_{21} &&= e_{2} e_{1},\\
\D_4 &= 6 s_{22} + s_{31} &&= e_{2} e_{1}^{2} + 5 e_{2}^{2},\\
\D_5 &= 20 s_{32} + s_{41} &&= e_{2} e_{1}^{3} + 18 e_{2}^{2} e_{1},\\
\D_6 &= 110 s_{33} + 50 s_{42} + s_{51} &&= e_{2} e_{1}^{4} + 47 e_{2}^{2} e_{1}^{2} + 61 e_{2}^{3},\\
\D_7 &= 700 s_{43} + 112 s_{52} + s_{61} &&= e_{2} e_{1}^{5} + 108 e_{2}^{2} e_{1}^{3} + 479 e_{2}^{3} e_{1},\\
\D_8 &= 4270 s_{44} + 3122 s_{53} + 238 s_{62} + s_{71} &&= e_{2} e_{1}^{6} + 233 e_{2}^{2} e_{1}^{4} + 2414 e_{2}^{3} e_{1}^{2} + 1385 e_{2}^{4}.
 \end{align*}
From \autoref{e_ribbons_pp}, we get
\begin{align}\label{e_ribbons_pp }
	 \D_n =\frac{e_2}{e_1^2} \sum_{\rho\in \mathcal{R}''_n} e_\rho,
\end{align}
Also observe that the Eulerian polynomials described by \autorefformula{formule_eulerian_pol} are homogeneous and $\gamma$-positive, hence Schur-positive (and unimodal). Small instances of their Schur and elementary functions expansions are as follows:
\begin{align*}
A_2 &= s_{2} + 3 \, s_{11} &&= e_{1}^{2} + 2\, e_{2},\\
A_3 &= s_{3}  + 10 \, s_{21}  &&= e_{1}^{3} + 8\, e_{2} e_{1},\\
A_4 &=  s_{4}+ 25\,  s_{31} + 40 \, s_{22} && = e_{1}^{4} + 22\, e_{2} e_{1}^{2} + 16\, e_{2}^{2},\\
A_5 &= s_{5} + 56 \, s_{41} + 245\,  s_{32}  &&= e_{1}^{5} + 52\, e_{2} e_{1}^{3} + 136\, e_{2}^{2} e_{1},\\
A_6 &=  s_{6} + 119 \, s_{51} + 1071 \, s_{42} +1225\,  s_{33} &&= e_{1}^{6} + 114\, e_{2} e_{1}^{4} + 720\, e_{2}^{2} e_{1}^{2} + 272\, e_{2}^{3},\\
A_7 &= s_{7} + 246 \, s_{61} + 4046 \, s_{52}  + 11326 \, s_{43} &&= e_{1}^{7} + 240\, e_{2} e_{1}^{5} + 3072\, e_{2}^{2} e_{1}^{3} + 3968\, e_{2}^{3} e_{1}.
\end{align*} 
Exploiting  \autoref{formule_eulerian_pol}, we may derive the formula (see \cite{Petersen2015})  
\begin{align}
    A_n = \sum_{2k< n} \sum_{j=0}^k (-1)^j \binom{n+2}{j} (k+1-j)^{n+1}  e_2^k\,p_{n-1-2k}
  \end{align}
which may be expanded as a polynomial in $\Z[e_1,e_2]$ using \autorefformula{p_to_e}. 
Moreover \autoref{qt_Foata_Schutzenberger_formula} (see \cite[Formula 1,3]{ShareshianWachs})  may be directly reformulated as
\begin{align}\label{formula_enumeration_no_double_descent}
    A_n   = \sum_{\sigma\in \DDesc'_{n}} e_{\rho(\sigma)},
\end{align}
where $\DDesc'_{n}$ stands for the set of permutations in $\Sym_n$ that have no ``double descents'' and no ``final descent''. Here $\rho(\sigma)=(c_1,c_2,\ldots, c_\ell)$ denotes the \unedefinition{ascent composition} of $n$ associated to $\sigma$, which is to say that $\sigma$ has ascents in position $(c_1+\ldots+c_j)_{j<\ell}$, with all $c_i\leq 2$ and $c_\ell=1$. 
\begin{figure}[ht]
\includegraphics[scale=.8]{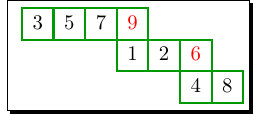}
\caption{The reading permutation of the above ribbon is $3579.126.48$, with descents marked by a point.}
\label{figure_reading_SYT_ribbon}
\end{figure}
These permutations also arise by reading standard young tableaux of ribbon shape $\rho$ (see \autorefname{Fig.}{figure_reading_SYT_ribbon}), with columns of length at most $2$, and a final column of length $1$ (see \autoref{schur_gamma_positivity}).  

The related family of binomial Eulerian polynomials (see \autoref{qt_binomial_eulerian}), affording the $e$-expansions
\begin{align*}
\A_0 &= 1 &&= 1,\\
\A_1 &= s_{1} &&= e_{1},\\
\A_2 &= s_{2} + 2 s_{11} &&=e_{1}^{2} + e_{2},\\
\A_3 &= s_{3} + 6 s_{21}  &&= e_{1}^{3} + 4 e_{2} e_{1},\\
\A_4 &=s_{4}  + 14 s_{31} + 18 s_{22}  &&= e_{1}^{4} + 11 e_{2} e_{1}^{2} + 5 e_{2}^{2},\\
\A_5 &=s_{5} + 30 s_{41} + 100 s_{32}  &&= e_{1}^{5} + 26 e_{2} e_{1}^{3} + 43 e_{2}^{2} e_{1},\\
\A_6 &=s_{6}  + 62 s_{51}+ 410 s_{42} + 410 s_{33}   &&= e_{1}^{6} + 57 e_{2} e_{1}^{4} + 230 e_{2}^{2} e_{1}^{2} + 61 e_{2}^{3},\\
\A_7&= s_{7} + 126 s_{61} + 1484 s_{52}  + 3500 s_{43}   &&= e_{1}^{7} + 120 e_{2} e_{1}^{5} + 990 e_{2}^{2} e_{1}^{3} + 906 e_{2}^{3} e_{1},
\end{align*}
with general $e$-expansion (see \cite[Thm 3.3]{MR4378084}):
\begin{align}
   \A_n = \sum_{\sigma\in\DDesc_n} e_{\rho(\sigma)}= \sum_{\rho\in \mathcal{R}_{n}} \# \SYT(\rho)\, e_\rho. \label{e_binomial_eulerian}
          \end{align}
The first equality is just a reformulation of  \autorefformula{qt_binomial_eulerian_formula} in terms of ascent composition of $\sigma$. In turn, the second equality (see \autoref{schur_gamma_positivity}) coins this in terms of
\unedefinition{Standard Young Tableaux}. We denote by $\SYT(\rho)$ the set of such tableaux have skew-shape $\rho$, with $\rho$ running over the set $\mathcal{R}_{n}$ of $n$-\unedefinition{cell ribbons with columns of length at most} $2$.  

%

\subsection{Some $\gamma$-anti-positive examples}
Recalling that \autorefformula{e_int_anti_positive} insures the $\gamma$-anti-positivity of $\eint{n}$. 
\autorefname{Lemma}{lem_gamma_anti} also implies the $\gamma$-anti-positivity of the $(q,t)$-factorials. Small instances are as follows: 
 \begin{align*}
\efact{1} &= 1,\\
\efact{2} &= e_{1},\\
\efact{3} &= e_{1}^{3} - e_{2} e_{1},\\
\efact{4} &= e_{1}^{6} - 3 e_{2} e_{1}^{4} + 2 e_{2}^{2} e_{1}^{2},\\
\efact{5} &= e_{1}^{10} - 6 e_{2} e_{1}^{8} + 12 e_{2}^{2} e_{1}^{6} - 9 e_{2}^{3} e_{1}^{4} + 2 e_{2}^{4} e_{1}^{2},\\
\efact{6} &= e_{1}^{15} - 10 e_{2} e_{1}^{13} + 39 e_{2}^{2} e_{1}^{11} - 75 e_{2}^{3} e_{1}^{9} + 74 e_{2}^{4} e_{1}^{7} - 35 e_{2}^{5} e_{1}^{5} + 6 e_{2}^{6} e_{1}^{3}.
 \end{align*}
Observe that the $(q,t)$-Catalan of \autorefformula{def_qt_catalan_formula} are not always Schur-positive but are $\gamma$-anti-positive:
\begin{align*}
&\Cat_2 = -s_{11} + s_{2}&&= e_{1}^{2} - 2 e_{2},\\
&\Cat_3 = s_{42} - s_{51} + s_{6}&&= e_{1}^{6} - 6 e_{2} e_{1}^{4} + 10 e_{2}^{2} e_{1}^{2} - 3 e_{2}^{3},\\
&\Cat_4 = s_{66} - s_{75} + s_{84} + s_{(10,2)} - s_{(11,1)} + s_{(12)}&&= e_{1}^{12} - 12 e_{2} e_{1}^{10} + 55 e_{2}^{2} e_{1}^{8} - 119 e_{2}^{3} e_{1}^{6}\\
              &&&\qquad\qquad \qquad+ 121 e_{2}^{4} e_{1}^{4} - 50 e_{2}^{5} e_{1}^{2} + 6 e_{2}^{6}.
\end{align*}
The different $(q,t)$-Catalan given by the \autorefname{Recurrence}{def_qt_catalan_recurrence} are not Schur-positive either, but they are ``graded $\gamma$-anti-same-sign'':
\begin{align*}
&\CatRec_2 = s_{1} &&= e_{1},\\
&\CatRec_3 = s_{11} + s_{3} &&= e_{1}^{3} - 2 e_{2} e_{1} + e_{2},\\
&\CatRec_4 = s_{22} + s_{31} - s_{32} + s_{33} + s_{41} + s_{6} &&= (e_{1}^{6} - 5 e_{2} e_{1}^{4} + 6 e_{2}^{2} e_{1}^{2})
									 \rouge{-}(- e_{2} e_{1}^{3} + 3 e_{2}^{2} e_{1}) + e_{2} e_{1}^{2}.
\end{align*}
We may recast the recurrence in the symmetric polynomial context as:
 \begin{align}\label{def_e_catalan_recurrence}
   \CatRec_{n+1}=\sum_{0\leq 2k\leq n+1} m_{(n-k,k)} \CatRec_{k} \CatRec_{n-k} ,
  \end{align}
and use \autorefformula{monomial_to_e} to coin this in terms of $e_1$ and $e_2$.
On the other hand, the more subtle $(q,t)$-Catalan related to the $\nabla$-operator are Schur-positive. For more on their $\gamma$-behaviour, see next section.
%
%
 
 Collecting some of our observations we have
   \begin{tcolorbox}[colback=azure!5!white]
\begin{lemma}\label{lemma_qt_binomials_semimagic} The following polynomials are all $\gamma$-anti-positive (lie in $\Gamma$):
\begin{itemize}
  \item The power-sum $p_n$, as well as $(-1)^b s_{ab}$ (see \ref{s_to_e}) and $(-1)^b m_{ab}$ (see \ref{monomial_to_e}),
  \item The $(q,t)$-integers $\eint{n}$ (see \ref{e_int_anti_positive}),
  \item The $(q,t)$-factorials $\efact{n}$,
  \item The $(q,t)$-binomials $\ebinom{n}{k}$, which satisfy recurrence \ref{e_recurrence_binom},
   \item The (simple) rational $(q,t)$-Catalan $\frac{1}{[a+b]} \ebinom{a+b}{a}$, with $a$ and $b$ co-prime.
   \end{itemize}
\end{lemma}
\end{tcolorbox}
 The next-to-last statement is  illustrated in \autoref{tab_semi_magic} (as expected with $\gamma$-anti-positivity, signs are tied with terms involving odd powers of $e_2$).
 \begin{table}[ht]
\renewcommand{\arraystretch}{1.2}
\begin{small}
\begin{tabular}{|c||l|l|l|l|l|}
\hline
$\textstyle \ebinom{n+k}{k}$ & $0$ & $1$ & $2$ & $3$ & $4$  \\[4pt]
\hline
\hline
$0$ & $1$ & $1$ & $1$ & $1$ & $1$ \\
\hline
$1$ & $1$ 
      & $e_{1}$ 
      & $e_{1}^{2} - e_{2}$ 
      & $e_{1}^{3} - 2 e_{2} e_{1}$ 
      & $e_{1}^{4} - 3 e_{2} e_{1}^{2} + e_{2}^{2}$ \\
\hline
$2$ & $1$ 
       & $e_{1}^{2} - e_{2}$ 
       & $e_{1}^{4} - 3 e_{2} e_{1}^{2} + 2 e_{2}^{2}$ 
       & $\begin{scriptstyle}\begin{aligned} e_{1}^{6} &- 5 e_{2} e_{1}^{4}+ 7 e_{2}^{2} e_{1}^{2} \\ 
      			     & - 2 e_{2}^{3}\end{aligned}\end{scriptstyle}$ 
       & $\begin{scriptstyle}\begin{aligned}\\[-5pt]  e_{1}^{8} &- 7 e_{2} e_{1}^{6}+ 16 e_{2}^{2} e_{1}^{4} \\ 
      			     &  - 13 e_{2}^{3} e_{1}^{2} + 3 e_{2}^{4}\end{aligned}\end{scriptstyle}$ \\[8pt]
\hline
$3$ & $1$ 
      & $e_{1}^{3} - 2 e_{2} e_{1}$ 
      & $\begin{scriptstyle}\begin{aligned} e_{1}^{6} &- 5 e_{2} e_{1}^{4}+ 7 e_{2}^{2} e_{1}^{2} \\ 
      			     & - 2 e_{2}^{3}\end{aligned}\end{scriptstyle}$ 
      & $\begin{scriptstyle}\begin{aligned} e_{1}^{9} &- 8 e_{2} e_{1}^{7}+ 22 e_{2}^{2} e_{1}^{5} \\ 
      			     &  - 23 e_{2}^{3} e_{1}^{3} + 6 e_{2}^{4} e_{1}\end{aligned}\end{scriptstyle}$ 
      & $\begin{scriptstyle}\begin{aligned}\\[-5pt] e_{1}^{12} &- 11 e_{2} e_{1}^{10} + 46 e_{2}^{2} e_{1}^{8}\\ 
      			     & - 90 e_{2}^{3} e_{1}^{6} + 81 e_{2}^{4} e_{1}^{4} \\ 
      			     &- 28 e_{2}^{5} e_{1}^{2} + 3 e_{2}^{6}\end{aligned}\end{scriptstyle}$  \\[14pt]
\hline
$4$ & $1$ 
      & $e_{1}^{4} - 3 e_{2} e_{1}^{2} + e_{2}^{2}$ 
      & $\begin{scriptstyle}\begin{aligned} e_{1}^{8} &- 7 e_{2} e_{1}^{6}+ 16 e_{2}^{2} e_{1}^{4} \\ 
      			     &  - 13 e_{2}^{3} e_{1}^{2} + 3 e_{2}^{4}\end{aligned}\end{scriptstyle}$
      & $\begin{scriptstyle}\begin{aligned} e_{1}^{12} &- 11 e_{2} e_{1}^{10} + 46 e_{2}^{2} e_{1}^{8}\\ 
      			     & - 90 e_{2}^{3} e_{1}^{6} + 81 e_{2}^{4} e_{1}^{4} \\ 
      			     &- 28 e_{2}^{5} e_{1}^{2} + 3 e_{2}^{6}\end{aligned}\end{scriptstyle}$ 
      & $\begin{scriptstyle}\begin{aligned}\\[-5pt] e_{1}^{16} & - 15 e_{2} e_{1}^{14} + 92 e_{2}^{2} e_{1}^{12}\\ 
      			     & - 296 e_{2}^{3} e_{1}^{10} + 533 e_{2}^{4} e_{1}^{8} \\ 
      			     &- 532 e_{2}^{5} e_{1}^{6}+ 277 e_{2}^{6} e_{1}^{4}\\ 
      			     & - 68 e_{2}^{7} e_{1}^{2} + 6 e_{2}^{8}\end{aligned}\end{scriptstyle}$  \\[20pt]
\hline
\end{tabular}
\end{small}
\caption{$e$-binomials are $\gamma$-anti-positive.}\label{tab_semi_magic}
 \end{table}
 
As follows from the upcoming discussion (see \autoref{sec_CGF}), the $(q,t)$-Narayana given by \autorefformula{def_q_narayana_formula} are $\gamma$-anti-positive. We get 
\begin{table}[ht]
\begin{small}
\begin{tabular}{|c|l|l|l|l|l|l|}
\hline
$ N_{nk}$ & $0$ & $1$ & $2$  \\[4pt] 
\hline
\hline
$1$ & $1$ & &\\
$2$ & $1$ & &\\
$3$ & $1$ & $e_{1}^{2} - e_{2}$ & \\
$4$ & $1$ & $e_{1}^{4} - 3 e_{2} e_{1}^{2} + 2 e_{2}^{2}$ & \\
$5$ & $1$ & $e_{1}^{6} - 5 e_{2} e_{1}^{4} + 7 e_{2}^{2} e_{1}^{2} - 2 e_{2}^{3}$ & $e_{1}^{8} - 7 e_{2} e_{1}^{6} + 17 e_{2}^{2} e_{1}^{4} - 16 e_{2}^{3} e_{1}^{2} + 4 e_{2}^{4}$ \\
$6$ & $1$ & $e_{1}^{8} - 7 e_{2} e_{1}^{6} + 16 e_{2}^{2} e_{1}^{4} - 13 e_{2}^{3} e_{1}^{2} + 3 e_{2}^{4}$ & $e_{1}^{12} - 11 e_{2} e_{1}^{10} + 47 e_{2}^{2} e_{1}^{8} - 97 e_{2}^{3} e_{1}^{6} + 97 e_{2}^{4} e_{1}^{4} - 41 e_{2}^{5} e_{1}^{2} + 6 e_{2}^{6}$ \\
\hline
\end{tabular}
\end{small}
  \caption{$e$-Narayana $N_{nk}$.}\label{narayana_table}
\end{table}
As we will see further below, many other $(q,t)$-polynomials encountered in combinatorics are $\gamma$-anti-positive. 

\section{Graded \texorpdfstring{$\gamma$}{gamma}-positivity and \texorpdfstring{$\gamma$}{gamma}-anti-positivity}

Observe that most of our previous examples are homogeneous, and thus are direct translations of usual $q$-analogues. The situation is quite different when one considers cases that are not homogeneous, for which setting $t=1$ is irreversible. In those instances, there is typically a \unedefinition{degree grading} that makes it possible to extend the various notions.

For the signless Stirling $e$-polynomials of the first kind, we evidently have the symmetric function recurrence (see \autorefname{Recurrence}{rec_stirling_first_kind})
\begin{align}\label{rec_e_stirling_first_kind}
    	 c_{nk}:=c_{(n-1,k-1)}+\eint{n-1}\,c_{(n-1,k)},
\end{align} 
giving graded $e$-expansions. Each graded component is $\gamma$-anti-positive, since this property is preserved by the recursion. 
For example, we get the values in \autoref{tab_e_stirling_first_kind}.
\begin{table}[ht]
\renewcommand{\arraystretch}{1.2}
\begin{small}
\begin{tabular}{|c||l|l|l|l|l|}
\hline
$\textstyle c_{nk}$ &  $1$ & $2$ & $3$ & $4$ & $5$ \\[4pt]
\hline
\hline
$2$ & $1$ & $1$ & &  &\\
\hline
$3$ & $e_{1}$ & $e_{1} + 1$ & $1$ & & \\
\hline
$4$ & $e_{1}^{3} - e_{2} e_{1}$ & $(e_{1}^{3} - e_{2} e_{1}) + (e_{1}^{2} - e_{2}) + e_{1}$ & $(e_{1}^{2} - e_{2}) + e_{1} + 1$ & $1$  &\\
\hline
$5$ & $e_{1}^{6} - 3 e_{2} e_{1}^{4} + 2 e_{2}^{2} e_{1}^{2}$ 
	& $\begin{scriptstyle}\begin{aligned} \\[-5pt]
	   (e_{1}^{6}& - 3 e_{2} e_{1}^{4}+ 2 e_{2}^{2} e_{1}^{2})  \\
	   & + (e_{1}^{5}- 3 e_{2} e_{1}^{3}  + 2 e_{2}^{2} e_{1})  \\
	   & + (e_{1}^{4} - 2 e_{2} e_{1}^{2}) \\
	   & + (e_{1}^{3} - e_{2} e_{1})\\[-8pt]
	   \end{aligned}\end{scriptstyle}$ 
	& $\begin{scriptstyle}\begin{aligned} \\[-5pt]
	   (e_{1}^{5} &- 3 e_{2} e_{1}^{3} + 2 e_{2}^{2} e_{1})\\
	   &+ (e_{1}^{4}  - 2 e_{2} e_{1}^{2})\\
	   & + (2 e_{1}^{3} - 3 e_{2} e_{1})\\
	   & + (e_{1}^{2} - e_{2}) + e_{1}\\[-8pt]
	   \end{aligned}\end{scriptstyle}$ 
	& $\begin{scriptstyle}\begin{aligned} (e_{1}^{3} &- 2 e_{2} e_{1})\\
	     & +(e_{1}^{2} - e_{2})\\
	     &  + e_{1} + 1\end{aligned}\end{scriptstyle}$ & $1$ \\
\hline
\end{tabular}
\end{small}
\caption{Graded sign-less Stirling $e$-polynomials of the first kind.}\label{tab_e_stirling_first_kind}
 \end{table}
Small instances of those of the second kind, satisfying the recurrence
\begin{align}\label{rec_e_stirling_second_kind}
  S_{nk}:=S_{(n-1,k-1)}+\eint{k}\, S_{(n-1,k)},
 \end{align}
are displayed in \autoref{tab_e_stirling_second_kind}. They are also $\gamma$-anti-positive.
\begin{table}[ht]
\renewcommand{\arraystretch}{1.2}
\begin{small}
\begin{tabular}{|c||l|l|l|l|l|ll}
\hline
$\textstyle S_{nk}$ &  $1$ & $2$ & $3$ & $4$ & $5$\\[4pt]
\hline
\hline
$2$ & $1$ & $1$ &&&\\
\hline
$3$ & $1$ & $e_{1} + 1$ & $1$ && \\
\hline
$4$ & $1$ & $e_{1}^{2} + e_{1} + 1$ & $e_{1}^{2} - e_{2} + e_{1} + 1$ & $1$& \\
\hline
$5$ & $1$ & $e_{1}^{3} + e_{1}^{2} + e_{1} + 1$ 
		& $\begin{scriptstyle}\begin{aligned} \\[-5pt]
	   	     (e_{1}^{4} &- 2 e_{2} e_{1}^{2}+ e_{2}^{2})\\
		     & + (e_{1}^{3} - e_{2} e_{1})\\
		     &  + (2 e_{1}^{2} - e_{2}) \\
		     & + e_{1} + 1\\[-8pt]
	  	 \end{aligned}\end{scriptstyle}$ 
		& $\begin{scriptstyle}\begin{aligned} 
			(e_{1}^{3} &- 2 e_{2} e_{1}) \\
		     & + (e_{1}^{2} - e_{2}) \\
		     & + e_{1} + 1
		    \end{aligned}\end{scriptstyle}$ 
		& $1$ \\
\hline
\end{tabular}
\end{small}
\caption{Graded Stirling $e$-polynomials of the second kind.}\label{tab_e_stirling_second_kind}
 \end{table}

For the version of $(q,t)$-Catalan polynomials of \autorefname{Definition}{def_qt_catalan_nabla}, the corresponding expansions of their image under the $\Anti{(-)}$-involution have the following expansions in the $e$-basis:
\begin{align*}
\Anti{\CatNab}_3  &=(e_{1}^{3} + 2 e_{2} e_{1})   + e_{2} ,\\
\Anti{\CatNab}_4  &=\left(e_{1}^{6} + 5 e_{2} e_{1}^{4} + 6 e_{2}^{2} e_{1}^{2} + e_{2}^{3}\right) 
			- \left(e_{2} e_{1}^{3} + 2 e_{2}^{2} e_{1}\right)  
			- \left(e_{2} e_{1}^{2} + e_{2}^{2}\right)  \\
\Anti{\CatNab}_5  &=\left(e_{1}^{10} + 9 e_{2} e_{1}^{8} + 28 e_{2}^{2} e_{1}^{6} + 35 e_{2}^{3} e_{1}^{4} + 15 e_{2}^{4} e_{1}^{2} +e_{2}^{5}\right)  \\
			&\qquad - \left(e_{2} e_{1}^{7} + 6 e_{2}^{2} e_{1}^{5} + 10 e_{2}^{3} e_{1}^{3} + 4 e_{2}^{4} e_{1}\right)  
			- \left(e_{2} e_{1}^{6} + 4 e_{2}^{2} e_{1}^{4} + 3 e_{2}^{3} e_{1}^{2}\right)   \\
			&\qquad - \left(e_{2} e_{1}^{5} + 4 e_{2}^{2} e_{1}^{3} + 4 e_{2}^{3} e_{1}\right)  
			+ \left(e_{2}^{2} e_{1}^{2} + e_{2}^{3}\right), 
\end{align*}
with terms grouped by degree. One observes that, up to a common sign depending on degree, each homogeneous component is $\gamma$-positive. However this nice property breaks down at $n=7$, with the degree $17$ component. Still, a twisted version of these, described below, does give rise to $\gamma$-positivity. To present this next step of our exploration, we consider the invertible\footnote{Up to applying $\Anti{(-)}$ and a simple sign change for each $s_{ab}$.} linear operator $F \longmapsto \widehat{F}$, on $\Lambda_2$, such that
\begin{align}\label{def_hat_operator}
    s_{ab} \longmapsto \widehat{s}_{ab}:=\lint{a-b+1}\cdot e_2^{b}.
\end{align}
As $\widehat{s}_{ab}$ is $\gamma$-positive, this operator  maps any graded Schur-positive $F$ to a graded $\gamma$-positive $\widehat{F}$.
In particular we get the graded $\gamma$-positivity of $\widehat{\CatNab}_n$. Here are some values
\begin{align*}
\widehat{\CatNab}_3 &= (e_{1}^{3} + 2 \, e_{1} e_{2}) + e_{2},\\
\widehat{\CatNab}_4 &=(e_{1}^{6} + 5 \, e_{1}^{4} e_{2} + 6 \, e_{1}^{2} e_{2}^{2} + e_{2}^{3})
					+ (e_{1}^{3} e_{2} + 2 \, e_{1} e_{2}^{2}) 
					+ (e_{1}^{2} e_{2} + e_{2}^{2}), \\
\widehat{\CatNab}_5 &= (e_{1}^{10} + 9 \, e_{1}^{8} e_{2} + 28 \, e_{1}^{6} e_{2}^{2} + 35 \, e_{1}^{4} e_{2}^{3} 
							+ 15 \, e_{1}^{2} e_{2}^{4} + e_{2}^{5})\\
					&\qquad\qquad+(e_{1}^{7} e_{2} + 6 \, e_{1}^{5} e_{2}^{2} + 10 \, e_{1}^{3} e_{2}^{3} + 4 \, e_{1} e_{2}^{4})\\
					&\qquad\qquad+(e_{1}^{6} e_{2} + 6 \, e_{1}^{4} e_{2}^{2} + 9 \, e_{1}^{2} e_{2}^{3} + 2 \, e_{2}^{4})\\ 
					&\qquad\qquad+(e_{1}^{5} e_{2} + 4 \, e_{1}^{3} e_{2}^{2} + 4 \, e_{1} e_{2}^{3})
						+(e_{1}^{2} e_{2}^{2} + e_{2}^{3}). 					
\end{align*}

\section{Schur \texorpdfstring{$\gamma$}{gamma}-positivity and \texorpdfstring{$\gamma$}{gamma}-anti-positivity}\label{schur_gamma_positivity}
To extend our discussion to notions considered in \cite{liao2026equivariant,ShareshianWachs}, we will now be dealing with expressions involving two kinds of symmetric functions. One involving the Frobenius transform of $\Sym_n$-representations (for which we have variables $\bx=(x_i)_i$), and the other associated to $(q,t)$-counting just as we have been doing up to now. To avoid 
explicit mention of these underlying variables, we may write\footnote{Being careful to conserve $(q,t)$-expressions on the left and $\bx$-expressions  on the right.} $\alpha\otimes \beta$ for $\alpha(q,t)\cdot \beta(\bx)$. 
 
As an introduction, consider the three contexts associated to the following families of \unedefinition{ribbons}\footnote{A connected skew-diagram that contain no $2\times 2$ square.} with columns of length at most $2$: 
\begin{equation}\label{containment}
	 \mathcal{R}''_{n}\subseteq \mathcal{R}'_{n} \subseteq \mathcal{R}_{n},
\end{equation}
where
\begin{itemize}
\item $\mathcal{R}_{n}$ is the set of $n$-cell ribbons having columns of length at most $2$,
\item $\mathcal{R}'_{n}$ stands for the subset of $\mathcal{R}_{n}$ of ribbons having its last column of length $1$, and
\item $\mathcal{R}''_{n}$ for the subset of $\mathcal{R}'_{nk}$ with the further restriction that the first column is also of length $1$.
\end{itemize}
It is handy to encode these ribbons as compositions of $n$ with parts of sizes $1$ or $2$ (see \autorefname{Fig.}{figure_ribbon}). We may then denote by $e_\rho$ the product of the $e_r$, for $r$ running the part of this composition.


Associated to each of the above families are the following \unedefinition{positive Schur}-$\gamma$ expansions (has may be found in  \cite{liao2026equivariant,ShareshianWachs}  with a slightly different notation):
\begin{align}
   \mathcal{Q}_n(\bx;q,t)  &= \sum_{\rho\in\mathcal{R}_{n}} e_\rho \otimes s_\rho(\bx),\label{Schur_binomial_eulerian}\\
   \mathcal{Q}'_n(\bx;q,t) &:=\sum_{\rho\in\mathcal{R}'_{n}} e_\rho \otimes s_\rho(\bx),\\
   \mathcal{Q}''_n(\bx;q,t) &:=\sum_{\rho\in\mathcal{R}''_{n}} e_\rho \otimes s_\rho(\bx).
 \end{align}
where $s_{\rho}(\bx)$ denotes the skew-Schur polynomials obtained by considering the ribbon $\rho$ as a skew-partition (see \autorefname{Fig.}{figure_ribbon}). Recalling the fact that $\scalar{ s_\rho(\bx)}{p_1^n}$ is equal to the number of standard Young tableaux of shape $\rho$, we get \autorefformula{e_binomial_eulerian}  from \autoref{Schur_binomial_eulerian}. 
 \begin{figure}[ht]
\includegraphics[scale=.7]{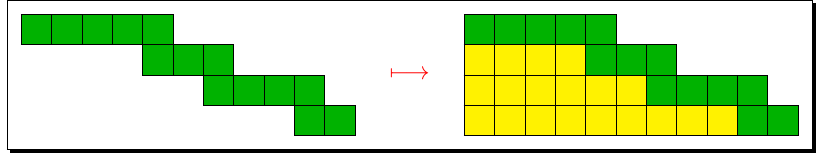}
\caption{Ribbon $11112121121 \in \mathcal{R}''_{(14,3)}$ $\mapsto$ skew-diagram  $(11,10,7,5)/(9,6,4)$.}
\label{figure_ribbon}
\end{figure}

\vskip-15pt\noindent
For example, we get the  ($e\,\otimes$\,Schur)-expansions:
\begin{align*}
&\mathcal{Q}''_5 =e_{1}^{4} \otimes s_{5} + 2 e_{2} e_{1}^{2} \otimes s_{41} +
					 2 e_{2} e_{1}^{2} \otimes s_{32}  ,\\
&\mathcal{Q}'_5 = e_{1}^{4} \otimes s_{5}+ 3 e_{2} e_{1}^{2} \otimes s_{41} + \left(2 e_{2} e_{1}^{2} + e_{2}^{2}\right)\otimes s_{32} 
												+ e_{2}^{2} \otimes s_{311}  +e_{2}^{2} \otimes s_{221}  ,\\
&\mathcal{Q}_5 =  e_{1}^{4} \otimes s_{5} + 4 e_{2} e_{1}^{2} \otimes s_{41} + \left(2 e_{2} e_{1}^{2} + 3 e_{2}^{2}\right)\otimes s_{32} +
				 3 e_{2}^{2} \otimes s_{311} +2 e_{2}^{2} \otimes s_{221}.
\end{align*}
These expressions are manifestly Schur\,$\otimes$\,Schur positive, which is also the case for $\mathcal{Q}_5-\mathcal{Q}'_5$ and $\mathcal{Q}'_5-\mathcal{Q}''_5$ (in view of the \autorefname{Inclusions}{containment}). Hence, we have unimodality of the associated $(q,t)$-enumeration for each irreducible component in these expressions. 
Other examples of \unedefinition{equivariant} $\gamma$-positivity, are considered in \cite{ShareshianWachs} and \cite{liao2026equivariant}. 

Also of interest are the $\widehat{(-)}$-images of the $(q,t)$-coefficients of Schur functions in $\nabla(e_n)$:
  $$\widehat{\nabla}(e_n(\bx)):=\sum_{\mu\vdash n} \widehat{ \scalar{ \nabla(e_n)}{s_\mu}}\, s_\mu(\bx),$$
which are graded Schur $\gamma$-positive. 

Once symmetrized, the $\Sym_n$-equivariant Chow polynomials of the braid matroid $B_n$ considered in \cite{kannan2026} is Schur $\gamma$-positive (as shown in  \cite{Ferroni2024Chow} for the specialization mentioned below). Some small values are as follows:
\begin{align*}
B_3 &= e_1\otimes s_3,\\[5pt]
B_4&= e_2 \otimes s_{22}
+e_2 \otimes s_{31}
+( e_1^2 + e_2 ) \otimes s_4,
\\[5pt]
B_5&=e_1e_2\otimes s_{221}
+3e_1e_2 \otimes s_{32}
+4e_1e_2\otimes s_{41}
+( e_1^3+ 2e_1e_2) \otimes s_{5},
\\[5pt]
B_6&=e_2^2 \otimes s_{2211}
+( 2 e_2 e_1^2+ 3e_2^2 ) \otimes s_{222}
+e_2^2 \otimes s_{3111}
+( 2 e_2 e_1^2+ 8 e_2^2 ) \otimes s_{321}\\
    &\qquad 
+( 2 e_2 e_1^2+ 4 e_2^2 ) \otimes s_{33}
+( e_2 e_1^2  + 5 e_2^2 ) \otimes s_{411}
+( 9 e_2 e_1^2 + 10 e_2^2 ) \otimes s_{42}\\
    &\qquad 
+( 7 e_2 e_1^2+ 7e_2^2 ) \otimes s_{51}
+( e_1^4+ 5e_2 e_1^2+ 3e_2^2 ) \otimes s_{6}.
\end{align*}
To compare this with \cite[Table 1]{kannan2026}; one sets $q=1$, hence $e_1=1+t$ and $e_2=t$.

\section{Probability of positivity}
When considering positivity phenomenons, one may wonder if these are rare events. The following results shed specific light on this, by measuring precisely how rarely positivity occurs in a random sample of objects considered. We may also consider that we here are comparing the relative ``strength'' of different properties.

Adapting a general result of \cite{Patrias}, we obtain the following.
   \begin{tcolorbox}[colback=azure!5!white]
\begin{proposition}
	The probability of Schur-positivity (unimodality) for a random\footnote{\ Taking an adequate limit, we may assume that coefficients are selected at random with uniform distribution.} monomial positive symmetric function (in $\N\{ m_{ab} \,|\, a+b=d\}$) is
	  $$((\lfloor d/2 \rfloor)!)^{-1}. $$
The probability of $\gamma$-positivity (or $\gamma$-anti-positivity) for a random Schur positive symmetric function (in $\N\{ s_{ab} \,|\, a+b=d\}$) is
         $$\prod_{2a+b= d}{\binom{b}{\lfloor b/2\rfloor}}^{-1}$$
\end{proposition}
\end{tcolorbox}
To get the adequate random distribution for Schur-positive symmetric polynomial (with a similar condition on random selection),  we exploit the identity
\begin{equation}
    \sum_{\newatop{\mu\vdash n}{ \ell(\mu)=2}} s_\mu =  h_{n-\lfloor n/2\rfloor}h_{\lfloor n/2\rfloor},
\end{equation}
and observe that
\begin{align*}
    \scalar{e_2^a e_1^b}{h_{n-\lfloor n/2\rfloor}h_{\lfloor n/2\rfloor}}
    		&=   \scalar{e_1^b}{(e_2^\perp)^ah_{n-\lfloor n/2\rfloor}h_{\lfloor n/2\rfloor}}\\
    		&=   \scalar{e_1^b}{ h_{b-\lfloor b/2\rfloor}h_{\lfloor b/2\rfloor}}\\
		&= \binom{b}{\lfloor b/2\rfloor}.
\end{align*}
Recall here that $e_2^\perp$ stands for the \unedefinition{skew-operator}, which is dual to multiplication by $e_2$ for the Hall scalar product, is such that $e_2^\perp s_{(a,b)}= s_{(a-1,b-1)}$ and $e_2^\perp s_{(a)}=0$.


\section{\texorpdfstring{$(q,t)$}{qt}-Cyclotomic polynomials}\label{section_cyclotomic_polynomials}

The usual cyclotomic polynomial $\varphi_n(q)$ is the minimal polynomial for primitive $n^{\rm th}$-roots of unity. They are thus irreducible over $\mathbb{Q}$. Properties of their homogenized versions, $\varphi_n:=\varphi_n(q,t)$, are obtained by a direct translation of those for the $q$-version. Leaving aside the case for $n=1$, we have 
\begin{align}\label{def_cyclo}
	\qtint{n}= \prod_{d | n,\ d\neq 1} \varphi_d(q,t).
\end{align}
Equivalently,
\begin{align*} 
	\varphi_n(q,t)  =\prod_{\newatop{d | n}{d\not=n}} ([d])^{\mu(n/d)},
\end{align*}
with $\mu(n)$ standing for the \unedefinition{M\"obius function} on integers.
From this, we get the $(q,t)$-symmetric polynomials:
\begin{align*}
\varphi_2(q,t)&=q + t,\\
\varphi_3(q,t)&=q^{2} + qt + t^2,&\varphi_4(q,t)&=q^{2} + t^2,\\
\varphi_5(q,t)&=q^{4} + q^{3}t + q^{2} t^2+ q t^3+ t^4,&\varphi_6(q,t)&=q^{2} - qt + t^2,\\
\varphi_7(q,t)&=q^{6} + q^{5} t+ q^{4}t^2 + q^{3}t^3 + q^{2}t^4 + q t^5+ t^6,&\varphi_8(q,t)&=q^{4} + 1,\\
\varphi_9(q,t)&=q^{6} + q^{3}t^3 + t^6,& \varphi_{10}(q,t)&=q^{4} - q^{3}t + q^{2}t^2 - qt^3 + t^4.
\end{align*} 
 As done before, we often omit mention of the variables, especially when we aim to underline that the expressions may be considered as abstract symmetric functions, and relations hold from the point of view of symmetric functions.
Small examples of $e$-expansion of the $(q,t)$-cyclotomic polynomials are:
\begin{small}
 \begin{align*}
&\varphi_{2} = e_{1}, &&
\varphi_{3} = e_{1}^{2} - e_{2},\\
&\varphi_{4} = e_{1}^{2} - 2 e_{2},&&
\varphi_{5} = e_{1}^{4} - 3 e_{2} e_{1}^{2} + e_{2}^{2},\\
&\varphi_{6} = e_{1}^{2} - 3 e_{2},&&
\varphi_{7} = e_{1}^{6} - 5 e_{2} e_{1}^{4} + 6 e_{2}^{2} e_{1}^{2} - e_{2}^{3},\\
&\varphi_{8} = e_{1}^{4} - 4 e_{2} e_{1}^{2} + 2 e_{2}^{2},&&
\varphi_{9} = e_{1}^{6} - 6 e_{2} e_{1}^{4} + 9 e_{2}^{2} e_{1}^{2} - e_{2}^{3},\\
&\varphi_{10} = e_{1}^{4} - 5 e_{2} e_{1}^{2} + 5 e_{2}^{2},&&
\varphi_{11} = e_{1}^{10} - 9 e_{2} e_{1}^{8} + 28 e_{2}^{2} e_{1}^{6} - 35 e_{2}^{3} e_{1}^{4} + 15 e_{2}^{4} e_{1}^{2} - e_{2}^{5},\\
&\varphi_{12} = e_{1}^{4} - 4 e_{2} e_{1}^{2} + e_{2}^{2},&&
\varphi_{13} = e_{1}^{12} - 11 e_{2} e_{1}^{10} + 45 e_{2}^{2} e_{1}^{8} - 84 e_{2}^{3} e_{1}^{6} + 70 e_{2}^{4} e_{1}^{4} - 21 e_{2}^{5} e_{1}^{2} + e_{2}^{6},\\
&\varphi_{14} = e_{1}^{6} - 7 e_{2} e_{1}^{4} + 14 e_{2}^{2} e_{1}^{2} - 7 e_{2}^{3},&&
\varphi_{15} = e_{1}^{8} - 9 e_{2} e_{1}^{6} + 26 e_{2}^{2} e_{1}^{4} - 24 e_{2}^{3} e_{1}^{2} + e_{2}^{4}.
 \end{align*}
 \end{small}\noindent
 One observes that all of these are $\gamma$-anti-positive.
This is a general fact that we will show in \autoref{section_lucas}. Since this property is closed under products, we conclude that we have $\gamma$-anti-positivity for all of the following (see \cite{Warnaar2010}):
 \begin{align}
	\efact{n} & = \prod_{d=2}^n \varphi_{d}^{\lfloor n/d\rfloor},\label{fact_cyclo_factorization}\\
	\ebinom{n}{k} &= \prod_{d=2}^n \varphi_{d}^{\lfloor n/d\rfloor -\lfloor k/d\rfloor-\lfloor (n-k)/d\rfloor}.\label{binom_cyclo_factorization}\\
	\frac{1}{[a+b]} \ebinom{a+b}{a} &= \prod_{d=2}^n \varphi_{d}^{\ \lfloor (a+b)/d\rfloor -\lfloor a/d\rfloor-\lfloor b/d\rfloor-\scharac(d\,|\,(a+b))},\label{formule_qtcat_cyclo}
\end{align}
where $\charac(P)$ is equal to $1$, if $P$ is true, and $0$ otherwise. Also, assuming that $n$ factorizes as $2^a\,k$ with $k$ odd, we have:
\begin{align}
	p_n&= \prod_{d\,|\, k} \varphi_{(2^{a+1}d)}.\label{formule_power_sum}
\end{align}
From \autoref{qt_Partition_distinct} it follows that:
\begin{align}
     p_1 p_2 \cdots p_n 
     	&=\prod_{0\leq 2d+1\leq n} \frac{\eint{2n-2d}} {\eint{2d+1}}\nonumber\\
     	&=   \prod_{1\leq d\leq n} \varphi_{(2\,d)} ^{\lfloor(n+d)/2d\rfloor}.\label{formule_partitions_distinct_parts}
 \end{align}
We also have $\gamma$-anti-positivity of:
\begin{align}\label{schur_plet_h_phi}
	 \prod_{(i,j)\in \mu} \frac{\eint{N+i-j}}{\eint{\varepsilon_\mu(i,j)}}.
\end{align}
\begin{table}
\begin{tabular}{|c||l|c|}
\hline
Group $W_n$& degrees $d_i$, $1\leq i\leq n$& Coxeter number \\
\hline
\hline
$A_n$ & $2, 3, 4, \ldots, n+1$ & $n+1$ \\
$B_n/C_n$ & $2, 4, 6,\ldots, 2n$ & $2n$ \\
$D_n$ & $2, 4, 6, \ldots, 2(n-1),n$ & $2(n-1)$ \\
$E_6$ & $2, 5, 6, 8, 9, 12$ & $12$ \\
$E_7$ & $2, 6, 8, 10, 12, 14, 18$ & $18$ \\
$E_8$ & $2, 8, 12, 14, 18, 20, 24, 30$ & $30$ \\
$F_4$ & $2, 6, 8, 12$ & $12$ \\
$G_2$ & $2, 6$ & $6$ \\
\hline
\end{tabular}
\caption{Crystallographic Coxeter groups}\label{tab_coxeter}
\end{table}
Furthermore, for each crystallographic Coxeter groups $W$   (listed in \autoref{tab_coxeter}), we have the $W$-Catalan:
\begin{align}\label{formule_WCat}
		\Cat_\WCoxeter(a)&:=\prod_{i=1}^n \frac{[e_i+a]}{[e_i+1]},
\end{align}
where $e_1<e_2<\ldots <e_n$ are the exponents of $W$, and $a$ is relatively prime to the Coxeter number of $W$.

 \subsection{Classical properties of cyclotomic polynomials}\label{sec_classical_prop_cyclo}
Exploiting plethysm, we have the following properties.
   \begin{tcolorbox}[colback=azure!5!white]
  \begin{lemma}\label{lemma_prop_cyclo}
When $a$ is a prime that divides $n$,  we have
\begin{align}
	p_{a}\circ \varphi_{n}=\varphi_{an}
\end{align}
and when it does not, we have
\begin{align}
	p_{a}\circ \varphi_{n} = \varphi_{n}\, \varphi_{an} 
\end{align} 
It follows that $p_{\mu}\circ(-)$ sends a product of cyclotomic to a product of cyclotomic, for any partition $\mu$.
\end{lemma}
\end{tcolorbox}
For a prime $a$, we have the formulas
\begin{align}
	p_{(a^k)}\circ \varphi_{n} = \begin{cases}
        \varphi_{(a^k n)}  & \text{if $a$ divides $n$}, \\[4pt]
       \prod_{j=0}^k \varphi_{(a^jn)}  & \text{if $a$ does not divide $n$}.
\end{cases}
\end{align}
Thus, if $c_i$ is the multiplicity of the part $i$ in $\mu$, we get
\begin{align}
	p_{\mu}\circ \varphi_{n} = \prod_{(i,n)\neq 1} \varphi_{(i^{c_i} n)} \prod_{(i,n)= 1} \prod_{j=0}^i \varphi_{(i^jn)}.
\end{align}
If $a$ is an odd prime, we also have 
 \begin{align}
	e_1 \varphi_{2a}=\eint{a+1}-e_2\,\eint{a-1}.
\end{align}
Using \autorefname{Lemma}{lemma_e_int}, it follows that (see \cite[Lem. 5.4]{sagan2020lucas-17f})
\begin{align}
 \varphi_{2a} =  \sum_{j} \left[\binom{a-j}{j}  -\binom{a-1-j}{j-1}\right]e_1^{a-2j-1}(-e_2)^j.
\end{align}

\section{Cyclotomic generating functions}\label{sec_CGF}
Directly adapting to our context the notions developed in \cite{Billey2024} and  \cite{Gatzweiler}, we consider the graded monoid $\Phi=(\Phi_d)_d$ of polynomials in $\Lambda_{2}$ (graded by degree) that may be expressed as a product:
\begin{align}
    f =  (-e_2)^b\varphi_{c_1} \cdots   \varphi_{c_k}.
\end{align}
The mostly ``technical'' reason for the signed $(-e_2)$ will become clear in the sequel.
Observe that \autorefname{Lemma}{lem_lucas_atom_gamma_positive} implies that  $\Phi_d\subseteq {\Gamma_d}$, and that all of its elements are $\gamma$-anti-positive.
\begin{table}
\begin{small}
 \renewcommand{\arraystretch}{1.2}
\begin{tabular}{|c||ll|}
\hline
Degree &\qquad $\varphi_n$ \quad with \quad $\deg(\varphi_n)=d$ &\\
\hline
\hline
$d=2$ & $\varphi_{3} = e_{1}^{2} - e_{2}$, 
	& $\varphi_{4} = e_{1}^{2} - 2 e_{2}$, \\
	& $\varphi_{6} = e_{1}^{2} - 3 e_{2}$, &\\
\hline
$d=4$ & $\varphi_{5} = e_{1}^{4} - 3 e_{2} e_{1}^{2} + e_{2}^{2}$, 
	& $\varphi_{8} = e_{1}^{4} - 4 e_{2} e_{1}^{2} + 2 e_{2}^{2}$,\\ 
	& $\varphi_{10} = e_{1}^{4} - 5 e_{2} e_{1}^{2} + 5 e_{2}^{2}$, 
	& $\varphi_{12} = e_{1}^{4} - 4 e_{2} e_{1}^{2} + e_{2}^{2}$, \\
\hline
$d=6$ & $\varphi_{7} = e_{1}^{6} - 5 e_{2} e_{1}^{4} + 6 e_{2}^{2} e_{1}^{2} - e_{2}^{3}$, 
	& $\varphi_{9} = e_{1}^{6} - 6 e_{2} e_{1}^{4} + 9 e_{2}^{2} e_{1}^{2} - e_{2}^{3}$,\\
	& $\varphi_{14} = e_{1}^{6} - 7 e_{2} e_{1}^{4} + 14 e_{2}^{2} e_{1}^{2} - 7 e_{2}^{3}$, 
	& $\varphi_{18} = e_{1}^{6} - 6 e_{2} e_{1}^{4} + 9 e_{2}^{2} e_{1}^{2} - 3 e_{2}^{3}$, \\
\hline
$d=8$ & $\varphi_{15} = e_{1}^{8} - 9 e_{2} e_{1}^{6} + 26 e_{2}^{2} e_{1}^{4} - 24 e_{2}^{3} e_{1}^{2} + e_{2}^{4}$,&\\ 
	& $\varphi_{16} = e_{1}^{8} - 8 e_{2} e_{1}^{6} + 20 e_{2}^{2} e_{1}^{4} - 16 e_{2}^{3} e_{1}^{2} + 2 e_{2}^{4}$,&\\ 
	& $\varphi_{20} = e_{1}^{8} - 8 e_{2} e_{1}^{6} + 19 e_{2}^{2} e_{1}^{4} - 12 e_{2}^{3} e_{1}^{2} + e_{2}^{4}$,&\\ 
	& $\varphi_{24} = e_{1}^{8} - 8 e_{2} e_{1}^{6} + 20 e_{2}^{2} e_{1}^{4} - 16 e_{2}^{3} e_{1}^{2} + e_{2}^{4}$,&\\ 
	& $\varphi_{30} = e_{1}^{8} - 7 e_{2} e_{1}^{6} + 14 e_{2}^{2} e_{1}^{4} - 8 e_{2}^{3} e_{1}^{2} + e_{2}^{4}$.
	 & \\
\hline
\end{tabular}
\end{small}
 \caption{$e$-Expansion for cyclotomic polynomials}\label{table_cyclotomic_e}
\end{table}
However, not all elements of $\Phi$ are $(q,t)$-positive\footnote{Besides the obvious $-e_2$.}:
\begin{align*}
&\varphi_{10} = q^{4} - q^{3} t + q^{2} t^{2} - q t^{3} + t^{4},&& 
\varphi_{6}^{2} = q^{4} - 2 q^{3} t + 3 q^{2} t^{2} - 2 q t^{3} + t^{4},\\
&\varphi_{4} \varphi_{6} = q^{4} - q^{3} t + 2 q^{2} t^{2} - q t^{3} + t^{4},&&
\varphi_{12} = q^{4} - q^{2} t^{2} + t^{4},
\end{align*}
leading us to consider the set $\Phi^+:=\Phi\cap \N[q,t]$. Clearly, $\Phi^+$ is closed under multiplication, and forms a graded submonoid of $\Gamma$, with $\Phi_n^+\subseteq {\Gamma}_n$ denoting its degree-$n$ homogeneous component. In particular, for any partitions $\lambda$ of $n$, we have
 $$h_\lambda:=\prod_{k\in \lambda} [k+1] \in \Phi^+_n.$$
A complete list of small degree elements of $\Phi^+$ is given in \autoref{table_CGF}. Many expressions considered in \autoref{section_cyclotomic_polynomials} lie in $\Phi^+$. 
\begin{table}
  \centering 
  \begin{small}
 \renewcommand{\arraystretch}{1.2}
\begin{tabular}{|c|llll|}
\hline
         & $\varphi_{3}$ 
         & $ = \eint{3} $ 
         & $ = e_{1}^{2} - e_{2}$ 
         &  $ =\bleu{s_{2}} $\\ 
$n=2$	
	&$\varphi_{2}^{2} $  
	& $ = \eint{2}^{2} $ 
	& $ = e_{1}^{2} $ 
	&  $ = \bleu{ s_{2} + s_{11}}$ \\
 	&$\varphi_{4}$ 
	& $ = {\eint{4}}/{\eint{2}}$ 
	& $ = e_{1}^{2} - 2 e_{2}$ 
	& $ =  s_{2}-s_{11} $\\ 
\hline
 	& $\varphi_{2} \varphi_{4}$ & $ = \eint{4}$ & $ = e_{1}^{3} - 2 e_{2} e_{1}$ &  $ = \bleu{ s_{3}}$\\
$n=3$	& $\varphi_{2} \varphi_{3}$ & $ = \eint{2} \eint{3}$ & $ = e_{1}^{3} - e_{2} e_{1}$ & $=\bleu{ s_{3}+ s_{21}} $\\
	&  $\varphi_{2}^{3}$ & $ = \eint{2}^{3}$ & $ = e_{1}^{3}$ & $ =  \bleu{ s_{3}+2 s_{21}} $ \\
	&$\varphi_{2} \varphi_{6}$ & $ = {\eint{6}}/{\eint{3}}$ & $ = e_{1}^{3} - 3 e_{2} e_{1}$ & $ =  s_{3}-s_{21} $\\
\hline
	&$e_2^2$ & $ = e_2^2$ & $ = e_{2}^{2}$ & $ =  \bleu{ s_{22}}$\\
	  &$\varphi_{5}$ & $ = \eint{5}$ & $ = e_{1}^{4} - 3 e_{2} e_{1}^{2} + e_{2}^{2}$ & $ =  \bleu{ s_{4}}$\\
	  &$\varphi_{2}^{2} \varphi_{4}$ & $ = \eint{4}\eint{2} $ & $ = e_{1}^{4} - 2 e_{2} e_{1}^{2}$ & $ =   \bleu{  s_{4}+s_{31}}$\\  
	  & $\varphi_{3}^{2}$ & $ = \eint{3}^{2}$ & $ = e_{1}^{4} - 2 e_{2} e_{1}^{2} + e_{2}^{2}$ & $ =  \bleu{ s_{4} + s_{31} + s_{22}}$\\ 
	  & $\varphi_{2}^{2} \varphi_{3}$ & $ =  \eint{3} \eint{2}^{2} $ & $= e_{1}^{4} - e_{2} e_{1}^{2}$ & $ =  \bleu{ s_{4} + 2 s_{31} + s_{22}}$ \\ 
$n=4$	  & $\varphi_{2}^{4}$ & $ = \eint{2}^{4}$ & $ = e_{1}^{4}$ & $ =   \bleu{ s_{4}+ 3 s_{31} + 2 s_{22} } $ \\
	 & $\varphi_{3} \varphi_{4}$ & $ = {\eint{3} \eint{4}}/{\eint{2}}$ & $ = e_{1}^{4} - 3 e_{2} e_{1}^{2} + 2 e_{2}^{2}$ & $ =  \rouge{ s_{4} + s_{22}}$\\
 	& $\varphi_{2}^{2} \varphi_{6}$ & $ = {\eint{2} \eint{6}}/{\eint{3}}$ & $ = e_{1}^{4} - 3 e_{2} e_{1}^{2}$ & $ =  s_{4}-s_{22} $\\  
	 & $\varphi_{8}$ & $ = {\eint{8}}/{\eint{4}}$ & $ = e_{1}^{4} - 4 e_{2} e_{1}^{2} + 2 e_{2}^{2}$ & $ =  s_{4}-s_{31} $\\
	 & $\varphi_{3} \varphi_{6}$ & $ = {\eint{6}}/{\eint{2}}$ & $ = e_{1}^{4} - 4 e_{2} e_{1}^{2} + 3 e_{2}^{2}$ & $ = s_{4} - s_{31} + s_{22}$\\
	 & $\varphi_{4}^{2}$ & $ = {\eint{4}^{2}}/{\eint{2}^{2}}$ & $ = e_{1}^{4} - 4 e_{2} e_{1}^{2} + 4 e_{2}^{2}$ & $ = s_{4} - s_{31} + 2 s_{22}$\\ 
\hline
\end{tabular}
\end{small}
  \caption{All elements of $\Phi_n^+$, with $2\leq n\leq 4$, with those that are log-concave coloured in blue, red if unimodal but not log concave, and black if not unimodal. }\label{table_CGF}
\end{table}
As discussed in \cite{Billey2024}, any homogeneous polynomial $f\in \Lambda_2^{(n)}$ that can be expressed as a quotient
\begin{align}\label{quotient_in_Phi}
    f_{\alpha\beta} = \prod_{i=1}^m \frac{\eint{a_i}}{\eint{b_i}}, 
\end{align}
for multisets $\alpha=\{a_i\}_{1\leq i\leq m}$ and $\beta=\{b_i\}_{1\leq i\leq m}$ of positive integers,
necessarily lies in $\Phi_n$. In fact, the $\varphi$-expansion is then:
\begin{align}
    f_{\alpha\beta} = \prod_{d} \varphi_d^{n_{\alpha\beta}(d)}, 
\end{align}
where $n_{\alpha\beta} (d)= \#\{ a_i\ |\ d\ \text{divides}\ a_i,\ 1\leq i\leq m\} - \#\{ b_i\ |\ d\ \text{divides}\ b_i,\ 1\leq i\leq m\}$. For sure, if $f_{\alpha\beta}$ is known to enumerate combinatorial objects (hence lies in $\N[q,t]$), it will belong to $\Phi^+_n$. In particular, from \autorefformula{formule_produit_qt_int}, we conclue that
\begin{align}
    p_k\circ [n] = \frac{[nk]}{[k]},
\end{align}
lies in $\Phi^+$.

For the poset structure on $\Phi^+_n$, obtained by setting $f\geq g$ if and only if $f-g \in \N[q,t]$, we seem to have
	$$  f \leq  \varphi_{2}^n = (q+t)^n, $$
for any $f$ in $\Phi^+_n$. In other words, in all $f\in\Phi^+_n$, the coefficient of $q^at^b$ is bounded above by the binomial coefficient $\binom{a+b}{a}$. Observe that this property is closed under products, as we clearly have $fg \leq  \varphi_{2}^{n+k}$ whenever $f \leq  \varphi_{2}^n$ and $g \leq  \varphi_{2}^k$.
For $n=4$, the bound may be seen to hold for all the $(q,t)$-expansions:
\begin{align*}
&e_{2}^{2} = q^{2} t^{2},
&&\varphi_{8} = q^{4} + t^{4},\\
&\varphi_{3} \varphi_{6} = q^{4} + q^{2} t^{2} + t^{4},
&&\varphi_{4}^{2} = q^{4} + 2 q^{2} t^{2} + t^{4},\\
&\varphi_{2}^{2} \varphi_{6} = q^{4} + q^{3} t + q t^{3} + t^{4},
&&\varphi_{5} = q^{4} + q^{3} t + q^{2} t^{2} + q t^{3} + t^{4},\\
&\varphi_{3} \varphi_{4} = q^{4} + q^{3} t + 2 q^{2} t^{2} + q t^{3} + t^{4},
&&\varphi_{2}^{2} \varphi_{4} = q^{4} + 2 q^{3} t + 2 q^{2} t^{2} + 2 q t^{3} + t^{4},\\
&\varphi_{3}^{2} = q^{4} + 2 q^{3} t + 3 q^{2} t^{2} + 2 q t^{3} + t^{4},
&&\varphi_{2}^{2} \varphi_{3} = q^{4} + 3 q^{3} t + 4 q^{2} t^{2} + 3 q t^{3} + t^{4},\\
&\varphi_{2}^{4} = q^{4} + 4 q^{3} t + 6 q^{2} t^{2} + 4 q t^{3} + t^{4};
\end{align*}
and has been checked for all $n\leq 20$.

\subsection{List of examples}
Among many classical instances, here is a partial list of enumeration formulas belonging to $\Phi^+$ (mostly extracted from \cite{Billey2024}): 
\begin{itemize}
\item {\it Catalan and Narayana like numbers and their analogues for Coxeter groups}.
\item  {\it Number of alternating sign matrices, or plane partitions}, since we have \autoref{qt_ASM}. These are not always Schur-positive.
\item  {\it Plane partitions fitting in a cube $a\times b\times c$}, in view of \autoref{qt_PPP_abc}. In terms of the $\varphi_n$'s, we have the expression:
\begin{align}
      \mathrm{PP}_{abc}  = \prod_{d=1}^{a+b+c-1} \varphi_d^{e_d},
 \end{align} 
 where (see Tufel) 
 		$$ e_d(a,b,c) = \sum_{i=0}^{a-1} \lfloor {(b+c+i)}/{d}\rfloor - \lfloor {(b+i)}/{d}\rfloor - \lfloor {(c+i)}/{d}\rfloor + \lfloor {i}/{d}\rfloor.$$
%
\item {\it Hook length formula, $\SYT_\mu$}.
\item {\it Length generating function of Weyl groups}.
\item {\it Hilbert series of the quotient of polynomial rings by HSOP}.
\item {\it Gaussian Posets $h$-polynomial}, for which e have the expression
 $$\prod_{i=1}^r \frac{\eint{a_i+m}}{\eint{a_i}}.$$
  See \cite[Page 332]{Proctor1984}.
 \end{itemize}
Moreover, a direct translation of \cite[Conj. 1.3]{Gatzweiler} states that
\begin{align}\label{question_fract_m_ai}
   \prod_{i=1}^n \left({[m+i]}/{[i]}\right)^{a_i} \in \Phi^+
\end{align}
when the expression considered is polynomial. Likewise, \cite[Conj. 1.2]{Gatzweiler} states that
 \begin{align}\label{question_fract_binom}
   {\textstyle{\ebinom{n}{k}}/{\ebinom{n}{\ell}}}\in  \Phi^+,
\end{align}
also when this quotient is a polynomial.

 \section{Lucas analogs}\label{section_lucas} 
Exploiting the isomorphism $\Anti{(-)}$, we may systematically transform previous formulas and identities into their $(e_1,e_2)$-\unedefinition{Lucas analogs}. Very similar expressions are considered\footnote{See for the history of the relevant notions.} in \cite{sagan2020lucas-17f}, with $s$ and $t$ (not our $t$)  in lieu of $e_1$ and $e_2$. This makes expressions non-homogenous at first glance, except if one declares that $t$ is of degree $2$. 
Our bivariate version of  \unedefinition{Lucas ``numbers''},  \unedefinition{Lucas factorials},  \unedefinition{Lucas binomials},  etc.
\begin{align}
	\lint{n}:=\Anti{\eint{n}},\qquad \lfact{n}:=\Anti{\efact{n}},\qquad \lbinom{n}{k}:=\Anti{\ebinom{n}{k}} , \qquad \text{etc.} 
\end{align}
makes the homogeneous aspect more prominent. As previously announced, relying on work of \cite{sagan2020lucas-17f}, we will show in  \autorefname{Proposition}{e_int_anti_positive} that all these analogs are  $\gamma$-positive as they are products of the $\gamma$-positive ``Lucas atom''  $\lvarphi{n}=\lvarphi{n}(q,t)$, which are themselves $\gamma$-positive. 

\subsection{The Lucas numbers}
From \autoref{e_int_anti_positive} we get that
   	\begin{align}
             \lint{n+1} =e_1\,\lint{n}+e_2\,\lint{n-1},\label{rec_e_int_anti}
   	\end{align}       
with initial conditions $\lint{0}=0$ and $\lint{1}=1$. Otherwise stated,
	\begin{align}
	   \lint{n+1} &= \sum_{k=0}^n \left(\sum_{j=0}^{n-k}\binom{k}{j}\binom{n-j}{k}\right) q^k t^{n-k}\\
	    	&= \sum_{i+2j=n} \binom{n-j}{j}e_1^{n-2j}e_2^j.
	\end{align}
Equivalently, we have the generating function
	\begin{align}
	    \sum_{n\geq 0} \{n+1\}\, z^n=\frac{1}{1-e_1\,z - e_2\, z^2},
	  \end{align}
so that
\begin{small}
\begin{align*}
&\lint{2} = e_{1},& 
&\lint{3} = e_{1}^{2} + e_{2},\\ 
&\lint{4} = e_{1}^{3} + 2 e_{2} e_{1},&
&\lint{5} = e_{1}^{4} + 3 e_{2} e_{1}^{2} + e_{2}^{2},\\ 
&\lint{6} = e_{1}^{5} + 4 e_{2} e_{1}^{3} + 3 e_{2}^{2} e_{1},&
&\lint{7} = e_{1}^{6} + 5 e_{2} e_{1}^{4} + 6 e_{2}^{2} e_{1}^{2} + e_{2}^{3},\\ 
&\lint{8} = e_{1}^{7} + 6 e_{2} e_{1}^{5} + 10 e_{2}^{2} e_{1}^{3} + 4 e_{2}^{3} e_{1},&
&\lint{9} = e_{1}^{8} + 7 e_{2} e_{1}^{6} + 15 e_{2}^{2} e_{1}^{4} + 10 e_{2}^{3} e_{1}^{2} + e_{2}^{4},\\ 
&\lint{10} = e_{1}^{9} + 8 e_{2} e_{1}^{7} + 21 e_{2}^{2} e_{1}^{5} + 20 e_{2}^{3} e_{1}^{3} + 5 e_{2}^{4} e_{1},& 
&\lint{11} = e_{1}^{10} + 9 e_{2} e_{1}^{8} + 28 e_{2}^{2} e_{1}^{6} + 35 e_{2}^{3} e_{1}^{4} + 15 e_{2}^{4} e_{1}^{2} + e_{2}^{5};
\end{align*}
\end{small}\noindent
as well as the matrix presentation (in view of \autorefformula{matrix_formula_int}):
\begin{align}\label{matrix_formula_anti_int}
      \begin{pmatrix} \lint{n+1}& \lint{n} \cdot e_2\\ \lint{n} & \lint{n-1}\cdot e_2\end{pmatrix} = \begin{pmatrix} e_1& e_2 \\ 1 & 0\end{pmatrix}^n,
\end{align}
As $e_1$ and $e_2$ are `independent'', we may use the specialization\footnote{Careful here. To conform with the notations of the paper cited, we use $t$ in a different way than previously.} $e_1=s$ and $e_2=t$ to obtain from $\lint{n}$ the Fibonacci polynomials of  \cite[section 3.1]{brittenham2016}. Thus many of the formulas herein correspond to those above. However, these become more ``understandable'' if one thinks that the $t$-variable is of degree $2$. This is clearly inherent in our approach. 

The classical \unedefinition{Fibonacci polynomials}, $F_n(x)=\lint{n}_{e_1=x,e_2=1}$. Another ``natural'' specialization corresponds to evaluating $e_1$ and $e_2$ at $q=t=1$, thus setting $e_1=2$ and $e_2=1$, so that we get the {Pell numbers}\footnote{That satisfy the recurrence $a_n=2 a_{n-1}+a_{n-2}$, with $a_0=0$ and $a_1=1$.} (see OEIS \href{https://oeis.org/A000129}{A000129}).
The evaluation of the odd cases $\lint{2k+1}_{e_1=2,e_1=1}$ corresponds to instances of Markov numbers. See \autoref{sec_Markov} for more on $(q,t)$-analogs of Markov numbers in general. Via other specializations, we recover several interesting notions, such as the  Fibonomials of \cite{BergeronCeballosKustner,Sagan2010} and related constructions.

\subsection{Typical Lucas-analogs}
Translating to our context  the recurrence of \cite[Prop. 3.1]{Bennett_2020}, we get the \unedefinition{Lucas-binomial} recurrence:
\begin{align}
    \lbinom{n}{k} = \lint{k+1}\lbinom{n-1}{k} + e_2\cdot \lint{n-k-1}\lbinom{n-1}{k-1}.
\end{align}
Moreover, in view of  \cite[Prop. 4.2]{Bennett_2020}, the \unedefinition{Lucas-Catalan} analogs of $\lcat_{n}$ are such that
\begin{align}
    \lcat_{n} =\lbinom{2n-1}{n-1} + e_2\cdot \lbinom{2n-1}{n-2}.
\end{align}
From \autorefformula{s_to_e} for the \unedefinition{Lucas-Schur}, we directly deduce 
    	\begin{align}\label{tilde_s_to_e} 
               \widetilde{s}_{ab} =\lint{a-b+1} \cdot (-e_2)^b  = (-1)^b\sum_{2j\leq a-b} \binom{a-b-j}{j}e_1^{a-b-2j}e_2^{b+j},
   	\end{align}   
or in matrix term:
    	\begin{align}\label{tilde_s_to_e_mat} 
  	\begin{pmatrix} \widetilde{s}_{(a+1,b)}& e_2\cdot \widetilde{s}_{ab}\\ 
     			    \widetilde{s}_{ab} & e_2\cdot  \widetilde{s}_{(a-1,b)}\end{pmatrix}
        =
             (-e_2)^b\cdot \begin{pmatrix} e_1& e_2 \\ 
     			    1 & 0\end{pmatrix}^{a-b+1}  
   	\end{align}   
For instance,
\begin{align*}
&-\widetilde{s}_{31} = 2 s_{22} + s_{31},
&&\widetilde{s}_{22} = s_{22},\\
&-\widetilde{s}_{41} = 4 s_{32} + s_{41},
&&\widetilde{s}_{32} = s_{32},\\
&-\widetilde{s}_{51} = 6 s_{33} + 6 s_{42} + s_{51},
&&\widetilde{s}_{42} = 2 s_{33} + s_{42},
&&-\widetilde{s}_{33} = s_{33},\\
&-\widetilde{s}_{61} = 16 s_{43} + 8 s_{52} + s_{61},
&&\widetilde{s}_{52} = 4 s_{43} + s_{52},
&&-\widetilde{s}_{43} = s_{43},\\
&-\widetilde{s}_{71} = 22 s_{44} + 30 s_{53} + 10 s_{62} + s_{71},
&&\widetilde{s}_{62} = 6 s_{44} + 6 s_{53} + s_{62},
&&-\widetilde{s}_{53} = 2 s_{44} + s_{53},
&&\widetilde{s}_{44} = s_{44}.
\end{align*}
The \unedefinition{Lucas-monomial} symmetric functions, are given by
\begin{align}\label{anti_monomial_to_e}
 	 \Anti{m}_{ab} = (-1)^b  \sum_{0\leq 2k \leq a-b}\frac{a-b}{a-b-k} \binom{a-b-k}{k} e_1^{a-b-2k}e_2^{k+b}.
\end{align}
The \unedefinition{Lucas-power-sum} $\Anti{p}_n$ may be obtained via a recurrence very similar to that for Lucas-numbers (\ref{rec_e_int_anti}), simply changing the initial conditions to $\Anti{p}_1=e_1$ and $\Anti{p}_2==e_1^2+2\,e_2$ (see \autorefname{Recurrence}{rec_e_power}), and then:
\begin{align}\label{rec_e_anti_power}
    \Anti{p}_n  = e_1\, \Anti{p}_{n-1} +e_2\,\Anti{p}_{n-2}.
\end{align}
Equivalently, we have the generating function
 \begin{align}\label{GF_e_anti_power}
    \sum_{n\geq 0} \Anti{p}_n x^n  =  \frac{e_1\, x-2e_2\,x^2}{1-e_1\,x-e_2\,x^2}
\end{align}
As illustrated by the following values, these polynomials are $e$-positive ({\sl aka} $\gamma$-positive):
\begin{align*}
\Anti{p}_2 &= e_{1}^{2} + 2 e_{2},\\
\Anti{p}_3 &= e_{1}^{3} + 3 e_{2} e_{1},\\
\Anti{p}_4 &= e_{1}^{4} + 4 e_{2} e_{1}^{2} + 2 e_{2}^{2},\\
\Anti{p}_5 &= e_{1}^{5} + 5 e_{2} e_{1}^{3} + 5 e_{2}^{2} e_{1},\\
\Anti{p}_6 &= e_{1}^{6} + 6 e_{2} e_{1}^{4} + 9 e_{2}^{2} e_{1}^{2} + 2 e_{2}^{3},\\
\Anti{p}_7 &= e_{1}^{7} + 7 e_{2} e_{1}^{5} + 14 e_{2}^{2} e_{1}^{3} + 7 e_{2}^{3} e_{1},\\
\Anti{p}_8 &= e_{1}^{8} + 8 e_{2} e_{1}^{6} + 20 e_{2}^{2} e_{1}^{4} + 16 e_{2}^{3} e_{1}^{2} + 2 e_{2}^{4}.
\end{align*}

Applying the isomorphism $\Anti{(-)}$ to \autorefname{Recurrences}{rec_e_stirling_first_kind} and \ref{rec_e_stirling_second_kind}, we clearly get recurrences for both kind of \unedefinition{Lucas-Stirling numbers}:
    \begin{align}
    	 \Anti{c}_{nk}&=\Anti{c}_{(n-1,k-1)}+\lint{n-1}\,\Anti{c}_{(n-1,k)},\qquad \text{and}\label{lucas_stirling_first_kind}\\	
  \Anti{S}_{nk}&=\Anti{S}_{(n-1,k-1)}+\lint{k}\, \Anti{S}_{(n-1,k)},\label{lucas_Stirling_second_kind}
 \end{align}
whose solutions are $\gamma$-positive.
For sure, we also have the \unedefinition{Lucas-Narayana} polynomials:
    \begin{align}
          \frac{1}{\lint{n}} \lbinom{n}{i} \lbinom{kn}{n-i-1},
    \end{align}
 as well as the \unedefinition{Lucas-$W$-Catalan} (see \autorefformula{formule_WCat}) for crystallographic Coxeter groups $W$:
\begin{align}
	\Anti{\Cat}_\WCoxeter(a):=\prod_{i=1}^n \frac{\lint{e_i+a}}{\lint{e_i+1}}.
\end{align}
Finally, we have tested that the following Lucas-analog of \autorefname{Conjecture}{conj_bergeron_vessenes} holds for all cases for which $ad=bc\leq 36$.
  \begin{tcolorbox}[colback=azure!5!white]
  \begin{conjecture}\label{conj_bergeron_vessenes_lucas}
    For all integers $a\leq b,c\leq d$ such that $ad=bc$, the polynomial
    \begin{align}\label{bergeron_vessenes_lucas}
         \lbinom{b+c}{b} -\lbinom{a+d}{a}
    \end{align} 
is Shur-positive, up to a global sign. 
  \end{conjecture}
    \end{tcolorbox}
  \noindent

\subsection{Lucas atoms}\label{subsection_lucas_atoms}
Let us translate in our setting the work of \cite{Alecci2025,Bennett_2020,sagan2020lucas-17f}.
From the analogous formula for cyclotomic polynomials, we obtain the following expression for the \unedefinition{Lucas atoms}:
\begin{align}\label{twisted_cyclo}
	\lvarphi{n}=\prod_{d | n, d\not=n} \{d\}^{\mu(n/d)},
\end{align}
which may be expanded as polynomials in $e_1$ and $e_2$.
Some small values are as follows:
\begin{small}
 \begin{align*}
&\lvarphi{2} = e_{1},&
&\lvarphi{3} = e_{1}^{2} + e_{2},\\
&\lvarphi{4} = e_{1}^{2} + 2 e_{2},&
&\lvarphi{5} = e_{1}^{4} + 3 e_{2} e_{1}^{2} + e_{2}^{2},\\
&\lvarphi{6} = e_{1}^{2} + 3 e_{2},&
&\lvarphi{7} = e_{1}^{6} + 5 e_{2} e_{1}^{4} + 6 e_{2}^{2} e_{1}^{2} + e_{2}^{3},\\
&\lvarphi{8} = e_{1}^{4} + 4 e_{2} e_{1}^{2} + 2 e_{2}^{2},&
&\lvarphi{9} = e_{1}^{6} + 6 e_{2} e_{1}^{4} + 9 e_{2}^{2} e_{1}^{2} + e_{2}^{3}\\
&\lvarphi{10} = e_{1}^{4} + 5 e_{2} e_{1}^{2} + 5 e_{2}^{2},&
&\lvarphi{11} = e_{1}^{10} + 9 e_{2} e_{1}^{8} + 28 e_{2}^{2} e_{1}^{6} + 35 e_{2}^{3} e_{1}^{4} + 15 e_{2}^{4} e_{1}^{2} + e_{2}^{5},\\
&\lvarphi{12} = e_{1}^{4} + 4 e_{2} e_{1}^{2} + e_{2}^{2},&
&\lvarphi{13} = e_{1}^{12} + 11 e_{2} e_{1}^{10} + 45 e_{2}^{2} e_{1}^{8} + 84 e_{2}^{3} e_{1}^{6} + 70 e_{2}^{4} e_{1}^{4} + 21 e_{2}^{5} e_{1}^{2} + e_{2}^{6},\\
&\lvarphi{14} = e_{1}^{6} + 7 e_{2} e_{1}^{4} + 14 e_{2}^{2} e_{1}^{2} + 7 e_{2}^{3},&&
\lvarphi{15} = e_{1}^{8} + 9 e_{2} e_{1}^{6} + 26 e_{2}^{2} e_{1}^{4} + 24 e_{2}^{3} e_{1}^{2} + e_{2}^{4}.
\end{align*}
\end{small}\noindent
Just as in \cite[Prop. 3]{Alecci2025} (see also \cite{Levy}), for all $n$ the polynomials $\lvarphi{n}$ are irreducible\footnote{This is not the case in $\QQ[q,t]$, as shown by $\lvarphi{7}(q,t)=(q^{3} + 5 q^{2} t + 6 q t^{2} + t^{3}) \cdot (q^{3} + 6 q^{2} t + 5 q t^{2} + t^{3})$. Observe that these factors are not symmetric in $q$ and $t$. Up to $n\leq 400$, the only other case which factorizes in this manner is $\lvarphi{30}(q,t)$.} in $\QQ[e_1,e_2]$. This is equivalent (via $\Anti{(-)}$) to the analogous irreducibility for $\varphi_{n}$. A direct adaptation of an argument of \cite{sagan2020lucas-17f}, given below, shows the following.  
   \begin{tcolorbox}[colback=azure!5!white]
\begin{proposition}\label{lem_lucas_atom_gamma_positive}
The Lucas atoms $\lvarphi{n}$  are  irreducible over $\QQ[e_1,e_2]$, $\gamma$-positive, and with negative real roots. Hence the polynomials $\varphi_{n}$ are $\gamma$-anti-positive. 
 \end{proposition}
 \end{tcolorbox}
\begin{proof}[\bf Proof]
We argue as in \cite{sagan2020lucas-17f}. Recall that primitive $n^{\text{th}}$-roots of unity come in pairs $\omega=a+b\,i$  and $\overline{\omega}$, for which we have $\omega\cdot \overline{\omega} = a^2+b^2=1$ and $2a=\omega\cdot \overline{\omega}$, with $-1< a <1$. Consider
\begin{align*}      (q-\omega\,t)   (q-\overline{\omega}\,t) &=q^2 - 2 a\,qt + t^2\\
     					                      &=(q+t)^2 \rouge{-} (2a+2)\,qt\\
     					                      &=e_1^2 \rouge{-} 2\,(a+1)\,e_2.
\end{align*}
The map $\Anti{(-)}$ transforms this polynomial into  $e_1^2 + 2(a+1) e_2=(q+t)^2 \rouge{+} 2(a+1)\,qt $,
which is positive over $\mathbb{R}$ since $0< a+1<2$. 
Thus $\lvarphi{n}(q,t)$    lies in  $\mathbb{R}_+[e_1,e_2]$, as it is a product of terms of the above form. Moreover, since $\lvarphi{n}$ lies in $\mathbb{Z}[e_1,e_2]$, we conclude that  $\lvarphi{n}\in \N[e_1,e_2]$. As ``roots'' of $(q+t)^2 +2(a+1)\,qt =q^2+2(a+2)qt+t^2$ are $-(a+2)\pm \sqrt{(a+2)^2-1}$, with $(a+2)>1$,  they are real and negative.
\end{proof}
 From the above proof of \autoref{lem_lucas_atom_gamma_positive}, we deduce the following formula for $\widetilde{\varphi}_{n}$. First, observe that each (quadratic) irreducible real factor of $\widetilde{\varphi}_{n}$ takes the form
 		$$e_1^2 + 2(a+1) e_2 = e_1^2+ e_2 {(\omega +1)^2}/{\omega},$$
for an adequate choice of a primitive $n^{\text{th}}$-root of unity $\omega$ (say with positive imaginary part). Hence we get
  	\begin{align}
	    \widetilde{\varphi}_{n}  = \prod_{\newatop{(k,n)=1}{2k<n}} \left(e_1^2+ \frac{(\omega_1^k +1)^2}{\omega_1^k} e_2\right),
	\end{align}
 with $\omega_1$ standing for the ``smallest'' primitive $n^{\text{th}}$-root of unity $e^{2 i \pi/n}$.

 \subsection{Properties of Lucas atoms}
 The effect of the \unedefinition{anti-plethysm} $p_n\,\Anti{\circ}\, (-)$ on elements of $\Lambda_2$, is explicitly obtained via the calculation  rules:
\begin{align*}
      &p_n\,\Anti{\circ}\, e_1= \lint{n+1}+e_2\,\lint{n-1}, &&p_n\,\Anti{\circ}\, e_2= e_2^n,\\
     &p_n\,\Anti{\circ}\,(a\,f+b\,g) =  a\,(p_n\,\Anti{\circ}\, f)+b\,(p_n\,\Anti{\circ}\, g)  &&p_n\,\Anti{\circ}\,(f\cdot g) =  (p_n\,\Anti{\circ}\, f)\cdot (p_n\,\Anti{\circ}\, g), 
\end{align*}
for any $f$ and $g$ in $\Lambda_2$, and scalars $a$ and $b$. This operation is extended linearly and multiplicatively to any $h\in \Lambda$. 
In other terms, anti-plethysm is obtained from plethysm via conjugation by $\Anti{(-)}$.
In particular, we get (see \cite[Prop. 4.6]{Ratliff2004}) 
\begin{align}\label{formule_pleth_int}
   (p_k\,\Anti{\circ}\, \lint{n}) =    \frac{\lint{nk} }{\lint{k}}
\end{align}
Simply translating to this context \autorefname{Lemma}{lemma_prop_cyclo}, we get the following.
If $c_i$ is the multiplicity of the part $i$ in $\mu$, then
\begin{align}
	p_{\mu}\,\Anti{\circ}\,  \Anti{\varphi}_{n} = \prod_{(i,n)\neq 1} \Anti{\varphi}_{(i^{c_i} n)} \prod_{(i,n)= 1} \prod_{j=0}^i \Anti{\varphi}_{(i^jn)}.
\end{align}
If $a$ is an odd prime, we also have 
 \begin{align}
	e_1\, \Anti{\varphi}_{2a}=\lint{a+1}+e_2\,\lint{a-1}.
\end{align}
and (see \cite[Lem. 5.4]{sagan2020lucas-17f})
\begin{align}
 \Anti{\varphi}_{2a} =  \sum_{j} \left[\binom{a-j}{j} +\binom{a-1-j}{j-1}\right]e_1^{a-2j-1}(e_2)^j.
\end{align}


\section{Anti-cyclotomic (Lucas atoms) generating functions}
Once again exploiting the involution $\Anti{(-)}$, we get the graded \unedefinition{Anti-Cyclotomic Generating Function} submonoid $\Anti{\Phi}$ of $\Anti{\Gamma}$, whose elements are of the form
    $$e_2^b\, \Anti{\varphi}_{c_1} \Anti{\varphi}_{c_2} \cdots \Anti{\varphi}_{c_\ell},$$
   for any sequences of integers $c_i\geq 2$.
As we have seen that all the $\lvarphi{n}$ are $\gamma$-positive ($e$-positive), Schur-positive, and with negative real roots, it follows that all elements of $\Anti{\Phi}$ have these properties. Thus we get the following from \cite[Prop 6.1 and Thm. 6.2]{Bennett_2020}, in clearly in the spirit of \cite{Billey2024,Gatzweiler}.
   \begin{tcolorbox}[colback=azure!5!white]
  \begin{proposition}\label{Anti_cyclotomic_generating _functions}
All elements of the anti-cyclotomic generating function monoid $\Anti{\Phi}$ are $\gamma$-positive, with negative real roots. Hence, they are log-concave and unimodal. The generating series for the number of elements of its graded components $\Anti{\Phi}_n$ is:
  \begin{align}
  	\sum_{n\geq 0} \#\Anti{\Phi}_n\, x^n = (1+x) (f(x)+f(-x))/2, \qquad \text{with}\qquad  f(x)=\prod_{n\geq 1} \frac{1}{1-x^{\phi(n)}},
  \end{align}
 where $\phi(n)$ is Euler $\phi$ function.
 \end{proposition}
  \end{tcolorbox}
\noindent
Thus, we get (see \cite{StanleyCyclo})
     $$\sum_{n\geq 0} \#\Anti{\Phi}_n\, x^n = 1 +x+ 5x^{2} + 5x^{3}+ 19x^{4} + 19x^{5}+ 59x^{6} + 59x^{7} + 165x^{8} + 165x^{9} + \ldots$$
\begin{table}
  \centering 
  \begin{small}
 \renewcommand{\arraystretch}{1.2}
\begin{tabular}{|llll|llll|}
\hline
&$\lvarphi{2}^{4}$ & $=\lint{2}^{4}$ & $=e_{1}^{4}$&
&$e_{2}^{2}$ & $=e_{2}^{2}$ & $=e_{2}^{2}$\\
&$e_{2} \lvarphi{2}^{2}$ & $=e_{2}\, \lint{2}^{2}$ & $=e_{2} e_{1}^{2}$&
&$\lvarphi{2}^{2} \lvarphi{3}$ & $=\lint{2}^{2} \lint{3}$ & $=e_{1}^{4} + e_{2} e_{1}^{2}$\\
&$\lvarphi{2}^{2} \lvarphi{4}$ & $=\lint{2} \lint{4}$ & $=e_{1}^{4} + 2 e_{2} e_{1}^{2}$&
&$\lvarphi{2}^{2} \lvarphi{6}$ & $={\lint{2} \lint{6}}/{\lint{3}}$ & $=e_{1}^{4} + 3 e_{2} e_{1}^{2}$\\
&$e_{2} \lvarphi{3}$ & $=e_{2}\, \lint{3}$ & $=e_{2} e_{1}^{2} + e_{2}^{2}$&
&$e_{2} \lvarphi{4}$ & $={e_{2}\, \lint{4}}/{\lint{2}}$ & $=e_{2} e_{1}^{2} + 2 e_{2}^{2}$\\
&$e_{2} \lvarphi{6}$ & $={e_{2}\, \lint{6}}/{\lint{2} \lint{3}}$ & $=e_{2} e_{1}^{2} + 3 e_{2}^{2}$&
&$\lvarphi{3}^{2}$ & $=\lint{3}^{2}$ & $=e_{1}^{4} + 2 e_{2} e_{1}^{2} + e_{2}^{2}$\\
&$\lvarphi{5}$ & $=\lint{5}$ & $=e_{1}^{4} + 3 e_{2} e_{1}^{2} + e_{2}^{2}$&
&$\lvarphi{3} \lvarphi{4}$ & $={\lint{3} \lint{4}}/{\lint{2}}$ & $=e_{1}^{4} + 3 e_{2} e_{1}^{2} + 2 e_{2}^{2}$\\
&$\lvarphi{12}$ & $={\lint{2} \lint{12}}/{\lint{4} \lint{6}}$ & $=e_{1}^{4} + 4 e_{2} e_{1}^{2} + e_{2}^{2}$&
&$\lvarphi{8}$ & $={\lint{8}}/{\lint{4}}$ & $=e_{1}^{4} + 4 e_{2} e_{1}^{2} + 2 e_{2}^{2}$\\
&$\lvarphi{3} \lvarphi{6}$ & $={\lint{6}}/{\lint{2}}$ & $=e_{1}^{4} + 4 e_{2} e_{1}^{2} + 3 e_{2}^{2}$&
&$\lvarphi{4}^{2}$ & $={\lint{4}^{2}}/{\lint{2}^{2}}$ & $=e_{1}^{4} + 4 e_{2} e_{1}^{2} + 4 e_{2}^{2}$\\
&$\lvarphi{10}$ & $={\lint{10}}/{\lint{2} \lint{5}}$ & $=e_{1}^{4} + 5 e_{2} e_{1}^{2} + 5 e_{2}^{2}$&
&$\lvarphi{4} \lvarphi{6}$ & $={\lint{4} \lint{6}}/{\lint{2}^{2} \lint{3}}$ & $=e_{1}^{4} + 5 e_{2} e_{1}^{2} + 6 e_{2}^{2}$\\
&$\lvarphi{6}^{2}$ & $={\lint{6}^{2}}/{\lint{2}^{2} \lint{3}^{2}}$ & $=e_{1}^{4} + 6 e_{2} e_{1}^{2} + 9 e_{2}^{2}$&&&&\\
\hline
\end{tabular}
\end{small}
  \caption{ The $19$ degree-$4$ elements of  $\Anti{\Phi}_4$ (see \href{https://oeis.org/A341711}{OEIS: A341711}).}\label{table_CGF_anti}
\end{table}
In particular, we have the following elements of $\Anti{\Phi}$ (see \autorefname{Equations}{def_cyclo} to \ref{formule_partitions_distinct_parts}):
\begin{align}
	\lint{n}&= \prod_{d | n,\ d\neq 1} \lvarphi{d}\\
	\lfact{n} &= \prod_{d=2}^n \lvarphi{d}^{\ \lfloor n/d\rfloor},\\
	\lbinom{n}{k} &= \prod_{d=2}^n \lvarphi{d}^{\ \lfloor n/d\rfloor -\lfloor k/d\rfloor-\lfloor (n-k)/d\rfloor},\\
 	\frac{1}{\lint{a+b}} \lbinom{a+b}{a} &= \prod_{d=2}^n \lvarphi{d}^{\ \lfloor (a+b)/d\rfloor -\lfloor a/d\rfloor-\lfloor b/d\rfloor-\scharac(d\,|\,(a+b))},
	\qquad (a,b)=1.\\
	\lp{n}&=\prod_{d\,|\, k} \lvarphi{(2^{a+1}d)},\label{formule_power_sum_anti}\\
     \lp{1} \lp{2} \cdots \lp{n}&=   \prod_{d=1}^{n} \lvarphi{(2\,d)} ^{\,\lfloor(n+d)/2d\rfloor}.\label{anti_formule_partitions_distinct_parts}
 \end{align}
As stated in \cite[Thm 1]{sagan2020lucas-17f}, 
   \begin{tcolorbox}[colback=azure!5!white]
  \begin{lemma}\label{lemma_product_anti_gen}
As in \autorefname{Definition}{quotient_in_Phi}, all polynomials that may be expressed in the form
\begin{align*}
    \widetilde{f}_{\alpha\beta} = \prod_{i=1}^m \frac{\lint{a_i}}{\lint{b_i}}, 
\end{align*}
for some multisets $\alpha=\{a_i\}_{1\leq i\leq m}$ and $\beta=\{b_i\}_{1\leq i\leq m}$ of positive integers,
necessarily lie in $\Anti{\Phi}$, and their $\widetilde{\varphi}$-expansion takes the form:
\begin{align*}
   \widetilde{f}_{\alpha\beta} = \prod_{d} \widetilde{\varphi}_d^{\ n_{\alpha\beta}(d)}, 
\end{align*}
with $n_{\alpha\beta} (d)= \#\{ a_i\ |\ d\ \text{divides}\ a_i,\ 1\leq i\leq m\} - \#\{ b_i\ |\ d\ \text{divides}\ b_i,\ 1\leq i\leq m\}$ positive.
 \end{lemma}
  \end{tcolorbox}
Thus, further instances include the polynomials 
\begin{align}
 \prod_{i=0}^{n-1}\frac{\lint{3i+1}}{\lint{n+i}},&&
     \prod_{i=1}^a\prod_{j=1}^b \prod_{k=1}^b\frac{\lint{i+j+k-1}}{\lint{i+j+k-2}},\qquad \text{and} &&  
    \prod_{1\leq i\leq j\leq k\leq r}  \frac{\lint{i+j+k-1}}{\lint{i+j+k-2}};
   \end{align}
the $\Anti{(-)}$-analogs\footnote{With various notations as before.}  
\begin{align}
 \widetilde{\mathrm{SYT}}_\mu &= \frac{\lfact{n}}{\prod_{c\in \mu}{\lint{\varepsilon_\mu(c)} }},\qquad \text{and}  
  &\widetilde{\mathrm{SSYT}}^{{\scriptstyle N}}_\mu  = e_2^{n(\mu)} \prod_{(i,j)\in \mu} \frac{\lint{N+i-j}}{\lint{\varepsilon_\mu(i,j)}};
 \end{align}
as well as polynomial cases of
\begin{align}
   \prod_{i=1}^n \left({\lint{m+i}}/{\lint{i}}\right)^{a_i} 
   \qquad  \text{and} \qquad  {{\lbinom{n}{k}}/{\lbinom{n}{\ell}}}.
\end{align}
Interesting questions arise when one looks for families of combinatorial objects counted by the coefficients of $q^it^j$ in these polynomials, as well as for coefficients of $e_1^a e_2^b$ ({\sl aka} $\gamma$-coefficients).


 
%

\section{Markov numbers (polynomials)}\label{sec_Markov}
As another form of exploration of the symmetric $(q,t)$-overview potential, let us consider a $(q,t)$-analog of Markov numbers that ``correspond'' to the $q$-analogues introduced in \cite{MorierGenoud}, with $(q,t)$-\unedefinition{Markov triples} $(X,Y,Z)$ satisfying  the symmetric function Markov equation: 
\begin{align}\label{Markov_equation}
 X^2 + Y^2 + Z^2 =\frac{[3]}{e_2} X\,Y\,Z +3 -\frac{[3]}{e_2}.
\end{align} 
Here, we consider that  $X={x}/{e_2^i}$, $Y={y}/{e_2^j}$, and $Z={z}/{e_2^k}$,  for homogeneous symmetric polynomials $x$, $y$ and $z$ in $\Lambda_2$ having respective degrees $2i$, $2j$ and $2k$. In other terms $X$, $Y$ and $Z$ are ``degree zero'' Laurent polynomials in $\Z[e_1,e_2^\pm]$. We may clearly go back and forth between $X$ and $x$, as convenient. 

A mostly direct translation of \cite[Prop. 2]{MorierGenoud} gives
   \begin{tcolorbox}[colback=azure!5!white]
\begin{proposition}\label{prop_MorierGenoud}
     Every $(q,t)$-Markov triple $(a,b,c)$ can be obtained from the triple $(1,1,e_2/e_2)$ by performing a sequence of \unedefinition{Vieta} transformation $(a,b,c)\mapsto (a,b,c')$, with
     \begin{align}\label{vieta_transformation}
     c' = \frac{\eint{3}}{e_2}\,a\,c-b
     \end{align}
 combined with permutations of $a$, $b$ and $c$.
\end{proposition}
\end{tcolorbox}
This implies (see \cite[Thm 1]{MorierGenoud}) that there is a unique\footnote{Modulo Markov's conjecture.} $(q,t)$-Markov triple that ``lifts'' a given classical Markov triple, hence we may associate a $(q,t)$-\unedefinition{Markov numbers} $\meint{a}$ to each Markov number $a$ (taking only the numerator of the solution). Thus, we get
\begin{align*}
&\meint{1} = e_{2},\\
&\meint{2} = e_{1}^{2} - 2 e_{2},\\
&\meint{5} = e_{1}^{4} - 3 e_{2} e_{1}^{2} + e_{2}^{2},\\
&\meint{13} = e_{1}^{6} - 4 e_{2} e_{1}^{4} + 3 e_{2}^{2} e_{1}^{2} + e_{2}^{3},\\
&\meint{29} = e_{1}^{8} - 6 e_{2} e_{1}^{6} + 12 e_{2}^{2} e_{1}^{4} - 9 e_{2}^{3} e_{1}^{2} + e_{2}^{4},\\
&\meint{34} = e_{1}^{8} - 5 e_{2} e_{1}^{6} + 6 e_{2}^{2} e_{1}^{4} + e_{2}^{3} e_{1}^{2} - 2 e_{2}^{4},\\
&\meint{89} = e_{1}^{10} - 6 e_{2} e_{1}^{8} + 10 e_{2}^{2} e_{1}^{6} - e_{2}^{3} e_{1}^{4} - 6 e_{2}^{4} e_{1}^{2} + e_{2}^{5},\\
&\meint{169} = e_{1}^{12} - 9 e_{2} e_{1}^{10} + 32 e_{2}^{2} e_{1}^{8} - 57 e_{2}^{3} e_{1}^{6} + 51 e_{2}^{4} e_{1}^{4} - 18 e_{2}^{5} e_{1}^{2} + e_{2}^{6}.
\end{align*}
Applying the $\Anti{(-)}$-operator, we get the $(q,t)$-\unedefinition{anti-Markov polynomials} $\mlint{a}:=\Anti{\eint{a}}\,\!^{{\scriptscriptstyle (M)}}$. Evaluating  both polynomials $(\meint{a},\mlint{a})$, at $q=t=1$, gives the sequence of pairs :
  $$ (2, 6), (5, 29), (29, 869), (169, 26041), (985, 780361), (5741, 23384789), (33461, 700763309), \ldots $$
Besides $\meint{2}$, among the above solutions and many other computed, those that appear to be $\gamma$-anti-positive seem to correspond $\meint{P_{2n+1}}$, for odd index of Pell numbers (see \href{https://oeis.org/A001653}{OEIS:A001653}).
Similar questions are also considered in \cite{NOVELLI2020102019}. 

Again up to a simple reformulation of \cite[Prop. 6]{MorierGenoud}, we get 
   \begin{tcolorbox}[colback=azure!5!white]
\begin{proposition}
  For all Markov number $a\geq 5$, the polynomial $\meint{a}$ is Schur-positive.
 \end{proposition}
 \end{tcolorbox}
 Small vales are as follows:
 \begin{align*}
&\meint{5} = s_{4},\\ 
&\meint{13} = s_{33} + s_{51} + s_{6},\\ 
&\meint{29} = s_{53} + 2 s_{62} + s_{71} + s_{8},\\ 
&\meint{34} = 2 s_{53} + s_{62} + 2 s_{71} + s_{8},\\ 
&\meint{89} = s_{55} + 3 s_{64} + 4 s_{73} + 3 s_{82} + 3 s_{91} + s_{(10)},\\ 
&\meint{169} = 2 s_{66} + 5 s_{75} + 6 s_{84} + 6 s_{93} + 5 s_{(10,2)} + 2 s_{(11,1)} + s_{(12)}.
\end{align*}
The $(q,t)$-expansions of $\meint{a}$ also appear to be log-concave, which has been checked by explicit calculations for all $a\leq 135137$. 
This pushes us to make the following conjecture (as we also checked it up to the same bound):
   \begin{tcolorbox}[colback=azure!5!white]
 \begin{conjecture}
  For all anti-Markov number $a\geq 2$, the polynomial $\mlint{a}$ is Schur-positive, and its $(q,t)$-expansion is log-concave.
 \end{conjecture}
 \end{tcolorbox}
Small values are:
\begin{align*}
&\mlint{2} = 3 s_{11}+s_{2},\\
&\mlint{5} = 6 s_{22} + 6 s_{31} + s_{4},\\
&\mlint{13} = 15 s_{33} + 24 s_{42} + 9 s_{51} + s_{6},\\
&\mlint{29} = 78 s_{44} + 127 s_{53} + 62 s_{62} + 13 s_{71} + s_{8},\\
&\mlint{34} = 48 s_{44} + 90 s_{53} + 51 s_{62} + 12 s_{71} + s_{8},\\
&\mlint{89} = 171 s_{55} + 345 s_{64} + 246 s_{73} + 87 s_{82} + 15 s_{91} + s_{(10)},\\
&\mlint{169} = 1364 s_{66} + 2687 s_{75} + 1926 s_{84} + 750 s_{93} + 167 s_{(10,2)} \\
     &\qquad\qquad\qquad + 20 s_{(11,1)} + s_{(12)}.
\end{align*}
Observe that the $\alpha$-map (see \href{https://people.mpim-bonn.mpg.de/zagier/files/doi/10.2307/2007348/fulltext.pdf}{Zagier}) to the Euclid tree corresponds to the half-degree of the symmetric Markov polynomials: $$\alpha(a)=1+\frac{1}{2}\deg(\meint{a}).$$

 \subsection{Another kind of \texorpdfstring{$(q,t)$}{qt}-deformation of Markov numbers}
A different kind of $q$-deformed ``squared'' Markov equation, considered in \cite{bittmannYildirim} \href{https://arxiv.org/pdf/2602.14802}{Bittmann-Jouteur-Oguz-Molander-Yildrin},  gives rise to the symmetric $(q,t)$-analog equation: 
  \begin{align}
     e_2 (X^2+Y^2+Z^2) + (\eint{3}-e_2) (XY +XZ +YZ) = 3 \,\eint{3} XYZ
  \end{align}  
Out of a given solution $(a,b,c)$, we may get another one $(a',b,c)$ by setting
 \begin{align}
      a' :=  3 \,\frac{\eint{3}}{e_2} bc - \frac{(\eint{3}-e_2)}{e_2} (b+c) -a
 \end{align}
 The entries of the resulting triples are degree-$0$ Laurent polynomials, with denominators that are pure powers of $e_2$. Thus, we need only describe their numerators in the $e$-basis (since they have matching degree). Small cases, in which we also display the corresponding Schur-expansion, are as follows:
 \begin{align*}
& 2^2, &&e_{1}^{2}, 
	\\ &&&s_{11} + s_{2} ,\\
& 5^2, &&2 e_{1}^{4} - 2 e_{2} e_{1}^{2} + e_{2}^{2}, 
         \\ &&&3 s_{22} + 4 s_{31} + 2 s_{4} ,\\
& 13^2, &&4 e_{1}^{6} - 6 e_{2} e_{1}^{4} + 2 e_{2}^{2} e_{1}^{2} + e_{2}^{3}, 
	\\ &&&11 s_{33} + 20 s_{42} + 14 s_{51} + 4 s_{6} ,\\
& 29^2, &&6 e_{1}^{8} - 14 e_{2} e_{1}^{6} + 14 e_{2}^{2} e_{1}^{4} - 6 e_{2}^{3} e_{1}^{2} + e_{2}^{4}, 
         \\ &&&37 s_{44} + 78 s_{53} + 64 s_{62} + 28 s_{71} + 6 s_{8} ,\\
& 34^2, &&8 e_{1}^{8} - 16 e_{2} e_{1}^{6} + 8 e_{2}^{2} e_{1}^{4} + e_{2}^{3} e_{1}^{2}, 
        \\ &&&49 s_{44} + 105 s_{53} + 88 s_{62} + 40 s_{71} + 8 s_{8} ,\\
& 169^2, &&18 e_{1}^{12} - 66 e_{2} e_{1}^{10} + 110 e_{2}^{2} e_{1}^{8} - 102 e_{2}^{3} e_{1}^{6} + 52 e_{2}^{4} e_{1}^{4} - 12 e_{2}^{5} e_{1}^{2} + e_{2}^{6}, 
        \\ &&&727 s_{66} + 1712 s_{75} + 1742 s_{84} + 1130 s_{93} + 488 s_{10.2} + 132 s_{11.1} + 18 s_{12.} .
   \end{align*}
Each affords a \unedefinition{palindromic factorizations} as a product of two $(q,t$)-polynomials in $\N[q,t]$, which is to say that one factor is the palindromic image of the other. For instance, we get
\begin{align*}
&2^2, &&(q + t)\cdot (t+q),\\
&5^2, &&(q^{2} + 2 q t + 2 t^{2}) \cdot (2 q^{2} + 2 q t + t^{2}),\\
&13^2, &&(2 q^{3} + 4 q^{2} t + 5 q t^{2} + 2 t^{3}) \cdot (2 q^{3} + 5 q^{2} t + 4 q t^{2} + 2 t^{3}),\\
&29^2, &&(2 q^{4} + 6 q^{3} t + 10 q^{2} t^{2} + 8 q t^{3} + 3 t^{4}) \cdot (3 q^{4} + 8 q^{3} t + 10 q^{2} t^{2} + 6 q t^{3} + 2 t^{4}),\\
&34^2, &&(2 \, q^{4} + 7 \, q^{3} t + 11 \, q^{2} t^{2} + 10 \, q t^{3} + 4 \, t^{4}) \cdot (4 \, q^{4} + 10 \, q^{3} t + 11 \, q^{2} t^{2} + 7 \, q t^{3} + 2 \, t^{4}).
\end{align*}
As we do not have a symmetric function factorization, the corresponding Lucas analogues are not necessarily squares. In fact, in most cases they are not. For instance, the $q=t=1$ evaluations of the $\Anti{(-)}$ images of our previous polynomials give (on the right are these polynomial images):
\begin{align*}
&\bleu{2}^2&&\tilde{\mapsto}&&4, && \bleu{e_{1}^{2}},\\
&\bleu{5}^2&&\tilde{\mapsto} &&41, &&\bleu{2 e_{1}^{4} + 2 e_{2} e_{1}^{2} + e_{2}^{2}},\\
&13^2&&\tilde{\mapsto} &&359, &&4 e_{1}^{6} + 6 e_{2} e_{1}^{4} + 2 e_{2}^{2} e_{1}^{2} - e_{2}^{3},\\
&\bleu{29}^2&&\tilde{\mapsto} &&2681, &&\bleu{6 e_{1}^{8} + 14 e_{2} e_{1}^{6} + 14 e_{2}^{2} e_{1}^{4} + 6 e_{2}^{3} e_{1}^{2} + e_{2}^{4}},\\
&34^2&&\tilde{\mapsto} &&3196, &&8 e_{1}^{8} + 16 e_{2} e_{1}^{6} + 8 e_{2}^{2} e_{1}^{4} - e_{2}^{3} e_{1}^{2},\\
&\bleu{169}^2&&\tilde{\mapsto} &&176881, &&\bleu{18 e_{1}^{12} + 66 e_{2} e_{1}^{10} + 110 e_{2}^{2} e_{1}^{8} + 102 e_{2}^{3} e_{1}^{6} + 52 e_{2}^{4} e_{1}^{4} + 12 e_{2}^{5} e_{1}^{2} + e_{2}^{6}}.
\end{align*}
Observe that $\gamma$-positivity (outlined in blue) again seems to align with (squares of) odd indexed Pell numbers.

%
%

 \subsection{Related questions}
Stanley (see \cite{StanleyCyclo}) considers \unedefinition{cyclotomic (multi)sets} $S\subset \N$ that have a generating function of the form
    $$G_S(x)=\frac{1}{1-x}-\sum_{j\in S} x^j = \frac{\prod_{i=1}^r (1-x^{a_i})} {\prod_{j=1}^t(1-x^{b_j})}$$
  for multisets $\alpha=\{a_i\}_{1\leq i\leq m}$ and $\beta=\{b_i\}_{1\leq i\leq m}$ of positive integers.
In particular, a finite set $S$ is cyclotomic, if and only if $1-(1-x) \sum_{j\in S} x^j$ can be expressed as a finite product of cyclotomic polynomials. These expressions may be homogenized just as we have been doing all along, and we may then consider their analogs in $\widetilde{\Phi}$. The notion of ``magic'' expansion of  polynomials (see \cite{avila2026luckmagicpitmanstanleypolytopes}) seems to bear many similarities with that of $\gamma$-expansions. It seems tantalizing to look for a symmetric $(q,t)$-analog context, that would  make natural many of their properties, just as we have shown to happen with $\gamma$-expansions.

\fancyhead{Bibliography}
\bibliographystyle{abbrv}    

\bibliography{bibliography}

@article{Alecci2025,
  title = {On alternative definition of {L}ucas atoms and their $p$-adic valuations},
  volume = {207},
  ISSN = {1436-5081},
  url = {http://dx.doi.org/10.1007/s00605-025-02087-w},
  DOI = {10.1007/s00605-025-02087-w},
  number = {2},
  journal = {Monatshefte f\"{u}r Mathematik},
  publisher = {Springer Science and Business Media LLC},
  author = {Alecci,  Gessica and Miska,  Piotr and Murru,  Nadir and Romeo,  Giuliano},
  year = {2025},
  month = May,
  pages = {175–196}
}

@article {MR2465404,
    AUTHOR = {Assaf, Sami H.},
     TITLE = {A generalized major index statistic},
   JOURNAL = {S\'em. Lothar. Combin.},
  FJOURNAL = {S\'eminaire Lotharingien de Combinatoire},
    VOLUME = {60},
      YEAR = {2008/09},
     PAGES = {Art. B60c, 14},
      ISSN = {1286-4889},
   MRCLASS = {05A15 (05A05 05E05 05E10 60A99 60C05)},
  MRNUMBER = {2465404},
MRREVIEWER = {Qing-Hu\ Hou},
}

@article{Athanasiadis2025,
  title = {On the real‐rootedness of the {E}ulerian transformation},
  volume = {111},
  ISSN = {1469-7750},
  url = {http://dx.doi.org/10.1112/jlms.70083},
  DOI = {10.1112/jlms.70083},
  number = {2},
  journal = {Journal of the London Mathematical Society},
  publisher = {Wiley},
  author = {Athanasiadis,  Christos A.},
  year = {2025},
  month = Feb 
}

@misc{AthanasiadisGamma,
  doi = {10.48550/ARXIV.1711.05983},
  url = {https://arxiv.org/abs/1711.05983},
  author = {Athanasiadis,  Christos A.},
  title = {Gamma-positivity in combinatorics and geometry},
  publisher = {arXiv},
  year = {2017},
  copyright = {arXiv.org perpetual,  non-exclusive license}
}

@misc{avila2026luckmagicpitmanstanleypolytopes,
      title={Luck and magic for {P}itman-{S}tanley polytopes and parking functions}, 
      author={Nicolas Avila and Luis Ferroni and Alejandro H. Morales},
      year={2026},
      eprint={2603.19194},
      archivePrefix={arXiv},
      primaryClass={math.CO},
      url={https://arxiv.org/abs/2603.19194}, 
}

@article {barrygammavectors,
    AUTHOR = {Barry, Paul},
     TITLE = {The {$\gamma$}-vectors of {P}ascal-like triangles defined by
              {R}iordan arrays},
   JOURNAL = {J. Integer Seq.},
  FJOURNAL = {Journal of Integer Sequences},
    VOLUME = {22},
      YEAR = {2019},
    NUMBER = {1},
     PAGES = {Art. 19.1.4, 22},
      ISSN = {1530-7638},
   MRCLASS = {11B83 (05A15 15B36)},
  MRNUMBER = {3917135},
MRREVIEWER = {Eric\ S.\ Egge},
}

@article{BergeronFoulkes, 
year = {2017}, 
title = {A $q$-Analog of {F}oulkes' Conjecture}, 
author = {Bergeron, Fran\c{c}ois}, 
journal = {The Electronic Journal of Combinatorics}, 
doi = {10.37236/6004}, 
number = {1}, 
volume = {24}, 
}

@article{BergeronCeballosKustner,
  title = {Elliptic and $q$-Analogs of the Fibonomial Numbers},
  ISSN = {1815-0659},
  url = {http://dx.doi.org/10.3842/SIGMA.2020.076},
  DOI = {10.3842/sigma.2020.076},
  journal = {Symmetry,  Integrability and Geometry: Methods and Applications},
  publisher = {SIGMA (Symmetry,  Integrability and Geometry: Methods and Application)},
  author = {Bergeron,  Nantel and Ceballos,  Cesar and K\"{u}stner,  Josef},
  year = {2020},
  month = Aug 
}

@book{BergeronBook,
	author = {Bergeron, Fran\c{c}ois},
	doi = {10.1201/b10583},
	isbn = {978-1-56881-324-0},
	mrclass = {05-01 (05Exx 13A50 20C30)},
	mrnumber = {2538310},
	mrreviewer = {Marc A. A. van Leeuwen},
	pages = {viii+221},
	publisher = {Canadian Mathematical Society, Ottawa, ON; A K Peters, Ltd., Wellesley, MA},
	series = {CMS Treatises in Mathematics},
	title = {Algebraic combinatorics and coinvariant spaces},
	url = {https://doi-org/10.1201/b10583},
	year = {2009}}

@article{Bennett_2020,
   title={Combinatorial Interpretations of {L}ucas Analogues of Binomial Coefficients and {C}atalan Numbers},
   volume={24},
   ISSN={0219-3094},
   url={http://dx.doi.org/10.1007/s00026-020-00500-9},
   DOI={10.1007/s00026-020-00500-9},
   number={3},
   journal={Annals of Combinatorics},
   publisher={Springer Science and Business Media LLC},
   author={Bennett, Curtis and Carrillo, Juan and Machacek, John and Sagan, Bruce E.},
   year={2020},
   pages={503–530} }

@article{Billey2024,
  title = {Cyclotomic Generating Functions},
  volume = {31},
  ISSN = {1077-8926},
  url = {http://dx.doi.org/10.37236/12687},
  DOI = {10.37236/12687},
  number = {4},
  journal = {The Electronic Journal of Combinatorics},
  publisher = {The Electronic Journal of Combinatorics},
  author = {Billey,  Sara and Swanson,  Joshua},
  year = {2024},
  month = Oct 
}

@misc{bittmannYildirim,
      title={A mirror deformation of {M}arkov numbers}, 
      author={L\'ea Bittmann and Perrine Jouteur and Ezgi Kantarcı Oğuz and Melody Molander and Emine Yıldırım},
      year={2026},
      eprint={2602.14802},
      archivePrefix={arXiv},
      primaryClass={math.CO},
      url={https://arxiv.org/abs/2602.14802}, 
}

@article {MR4378084,
    AUTHOR = {Br\"and\'en, Petter and Jochemko, Katharina},
     TITLE = {The {E}ulerian transformation},
   JOURNAL = {Trans. Amer. Math. Soc.},
  FJOURNAL = {Transactions of the American Mathematical Society},
    VOLUME = {375},
      YEAR = {2022},
    NUMBER = {3},
     PAGES = {1917--1931},
      ISSN = {0002-9947,1088-6850},
   MRCLASS = {52B05 (05A15 26C10 52B20)},
  MRNUMBER = {4378084},
MRREVIEWER = {Andr\'es\ R.\ Vindas-Mel\'endez},
       DOI = {10.1090/tran/8539},
       URL = {https://doi-org/10.1090/tran/8539},
}

@article {MR4259159,
    AUTHOR = {Br\"and\'en, Petter and Solus, Liam},
     TITLE = {Symmetric decompositions and real-rootedness},
   JOURNAL = {Int. Math. Res. Not. IMRN},
  FJOURNAL = {International Mathematics Research Notices. IMRN},
      YEAR = {2021},
    NUMBER = {10},
     PAGES = {7764--7798},
      ISSN = {1073-7928,1687-0247},
   MRCLASS = {12D10 (05A15 13F55 30C15 52B20)},
  MRNUMBER = {4259159},
MRREVIEWER = {Matthias\ Beck},
       DOI = {10.1093/imrn/rnz059},
       URL = {https://doi-org.proxy.bibliotheques.uqam.ca/10.1093/imrn/rnz059},
}

@misc{brittenham2016,
      title={Unimodality via alternating gamma vectors}, 
      author={Charles Brittenham and Andrew Carroll and T. Kyle Petersen and Connor Thomas},
      year={2016},
      eprint={1601.04979},
      archivePrefix={arXiv},
      primaryClass={math.CO},
      url={https://arxiv.org/abs/1601.04979}, 
}

@article{CaiReaddy,
title = {$q$-{S}tirling numbers: A new view},
journal = {Advances in Applied Mathematics},
volume = {86},
pages = {50-80},
year = {2017},
issn = {0196-8858},
doi = {https://doi.org/10.1016/j.aam.2016.11.007},
url = {https://www.sciencedirect.com/science/article/pii/S019688581630121X},
author = {Yue Cai and Margaret A. Readdy},}

@article{Corcino,
  doi = {10.5281/ZENODO.10083989},
  url = {https://zenodo.org/doi/10.5281/zenodo.10083989},
  author = {Corcino,  Roberto},
  title = {On $p,q$-binomial Coefficients},
  journal = {Integers: Electronic Journal Of Combinatorial Number Theory},
  volume = {8},
  number = {A29},  
  publisher = {Zenodo},
  year = {2008},
  copyright = {Creative Commons Attribution 4.0 International}
}

@article{Dey2020,
  title = {Gamma Positivity of the Descent Based {E}ulerian Polynomial in Positive Elements of Classical {W}eyl Groups},
  volume = {27},
  ISSN = {1077-8926},
  url = {http://dx.doi.org/10.37236/9037},
  DOI = {10.37236/9037},
  number = {3},
  journal = {The Electronic Journal of Combinatorics},
  publisher = {The Electronic Journal of Combinatorics},
  author = {Dey,  Hiranya Kishore and Sivasubramanian,  Sivaramakrishnan},
  year = {2020},
  month = Aug 
}

@article{Ferroni2024Chow,
  title = {Hilbert–Poincar\'e series of matroid Chow rings and intersection cohomology},
  volume = {449},
  ISSN = {0001-8708},
  url = {http://dx.doi.org/10.1016/j.aim.2024.109733},
  DOI = {10.1016/j.aim.2024.109733},
  journal = {Advances in Mathematics},
  publisher = {Elsevier BV},
  author = {Ferroni,  Luis and Matherne,  Jacob P. and Stevens,  Matthew and Vecchi,  Lorenzo},
  year = {2024},
  month = July,
  pages = {109733}
}

@article{Foata,
Author = {Dominique Foata and Marcel-Paul Schützenberger},
Title = {Th\'eorie G\'eom\'etrique des Polynômes Eul\'eriens},
Year = {2005},
Eprint = {arXiv:math/0508232},
Howpublished = {Lecture Notes in Mathematics, 138, Berlin, Springer-Verlag, 1970},
}

@article {Fouvry2018,
    AUTHOR = {Fouvry, \'Etienne and Levesque, Claude and Waldschmidt,
              Michel},
     TITLE = {Representation of integers by cyclotomic binary forms},
   JOURNAL = {Acta Arith.},
  FJOURNAL = {Acta Arithmetica},
    VOLUME = {184},
      YEAR = {2018},
    NUMBER = {1},
     PAGES = {67--86},
      ISSN = {0065-1036,1730-6264},
   MRCLASS = {11E76 (12E10)},
  MRNUMBER = {3826641},
MRREVIEWER = {Jing-Jing\ Huang},
       DOI = {10.4064/aa171012-24-12},
       URL = {https://doi-org.proxy.bibliotheques.uqam.ca/10.4064/aa171012-24-12},
}

@article{Gal2005,
  title = {Real Root Conjecture Fails for Five- and Higher-Dimensional Spheres},
  volume = {34},
  ISSN = {1432-0444},
  url = {http://dx.doi.org/10.1007/s00454-005-1171-5},
  DOI = {10.1007/s00454-005-1171-5},
  number = {2},
  journal = {Discrete and Computational Geometry},
  publisher = {Springer Science and Business Media LLC},
  author = {Gal,  Swiatoslaw R.},
  year = {2005},
  month = May,
  pages = {269–284}
}

@misc{Gatzweiler,
  doi = {10.48550/ARXIV.2603.22226},
  url = {https://arxiv.org/abs/2603.22226},
  author = {Gatzweiler,  Mona and Levicán-Santibáñez,  Fabián and Yoshida,  Atsuro},
  title = {Cyclotomic generating functions,  empty weighted complete intersections and positivity},
  publisher = {arXiv},
  year = {2026},
  copyright = {Creative Commons Attribution 4.0 International}
}

@incollection{haglund2015catalan,
  author    = {Haglund, James},
  title     = {{C}atalan Paths and $q,t$-Enumeration},
  booktitle = {Handbook of Enumerative Combinatorics},
  year      = {2015},
  publisher = {CRC Press},
  chapter   = {17},
  pages     = {497--531},
  url       = {https://www2.math.upenn.edu/~jhaglund/preprints/jh3.pdf}
}

@book{haglund,
	author = {Haglund, James},
	isbn = {978-0-8218-4411-3; 0-8218-4411-3},
	mrclass = {05E05 (05A05 05A30 33D52)},
	mrnumber = {2371044},
	mrreviewer = {Michael A. Zabrocki},
	note = {With an appendix on the combinatorics of Macdonald polynomials},
	pages = {viii+167},
	publisher = {American Mathematical Society, Providence, RI},
	series = {University Lecture Series},
	title = {The $(q,t)$-{C}atalan numbers and the space of diagonal harmonics},
	volume = {41},
	year = {2008}}

@article{Han2021,
  title = {Gamma-positivity of derangement polynomials and binomial Eulerian polynomials for colored permutations},
  volume = {182},
  ISSN = {0097-3165},
  url = {http://dx.doi.org/10.1016/j.jcta.2021.105459},
  DOI = {10.1016/j.jcta.2021.105459},
  journal = {Journal of Combinatorial Theory,  Series A},
  publisher = {Elsevier BV},
  author = {Han,  Bin},
  year = {2021},
  month = Aug,
  pages = {105459}
}

@article{Ji2024,
  title = {The binomial-{S}tirling–{E}ulerian polynomials},
  volume = {120},
  ISSN = {0195-6698},
  url = {http://dx.doi.org/10.1016/j.ejc.2024.103962},
  DOI = {10.1016/j.ejc.2024.103962},
  journal = {European Journal of Combinatorics},
  publisher = {Elsevier BV},
  author = {Ji,  Kathy Q. and Lin,  Zhicong},
  year = {2024},
  month = Aug,
  pages = {103962}
}

@misc{kannan2026,
      title={The $\mathbb{S}_n$-equivariant Chow polynomial of the braid matroid}, 
      author={Siddarth Kannan and Lukas Kühne},
      year={2026},
      eprint={2504.19829},
      archivePrefix={arXiv},
      primaryClass={math.AG},
      url={https://arxiv.org/abs/2504.19829}, 
}

@article {LerouxMedicis,
    AUTHOR = {de M\'edicis, Anne and Leroux, Pierre},
     TITLE = {A unified combinatorial approach for {$q$}- (and {$p,q$}-)
              {S}tirling numbers},
   JOURNAL = {J. Statist. Plann. Inference},
  FJOURNAL = {Journal of Statistical Planning and Inference},
    VOLUME = {34},
      YEAR = {1993},
    NUMBER = {1},
     PAGES = {89--105},
      ISSN = {0378-3758,1873-1171},
   MRCLASS = {05A30 (05E10 11B65 11B73)},
  MRNUMBER = {1209992},
MRREVIEWER = {Ira\ Gessel},
       DOI = {10.1016/0378-3758(93)90036-6},
       URL = {https://doi-org.proxy.bibliotheques.uqam.ca/10.1016/0378-3758(93)90036-6},
}

@article{Levy,
  title = {The Irreducible Factorization of {F}ibonacci Polynomials Over $\mathcal{Q}$},
  volume = {39},
  ISSN = {2641-340X},
  url = {http://dx.doi.org/10.1080/00150517.2001.12428710},
  DOI = {10.1080/00150517.2001.12428710},
  number = {4},
  journal = {The Fibonacci Quarterly},
  publisher = {Informa UK Limited},
  author = {Levy,  Dan},
  year = {2001},
  month = Aug,
  pages = {309–319}
}

@misc{liao2026equivariant,
      title={Equivariant gamma-positivity of matroid Chow rings}, 
      author={Hsin-Chieh Liao},
      year={2026},
      eprint={2408.00745},
      archivePrefix={arXiv},
      primaryClass={math.CO},
      url={https://arxiv.org/abs/2408.00745}, 
}

@article{Liu2025,
  title = {Some interlacing properties related to the {E}ulerian and derangement polynomials},
  volume = {162},
  ISSN = {0196-8858},
  url = {http://dx.doi.org/10.1016/j.aam.2024.102776},
  DOI = {10.1016/j.aam.2024.102776},
  journal = {Advances in Applied Mathematics},
  publisher = {Elsevier BV},
  author = {Liu,  Lily Li and Yan,  Xue},
  year = {2025},
  month = Jan,
  pages = {102776}
}

@article{Ma2024,
  title = {Positivity of {N}arayana polynomials and {E}ulerian polynomials},
  volume = {154},
  ISSN = {0196-8858},
  url = {http://dx.doi.org/10.1016/j.aam.2023.102656},
  DOI = {10.1016/j.aam.2023.102656},
  journal = {Advances in Applied Mathematics},
  publisher = {Elsevier BV},
  author = {Ma,  Shi-Mei and Qi,  Hao and Yeh,  Jean and Yeh,  Yeong-Nan},
  year = {2024},
  month = Mar,
  pages = {102656}
}

@book{Macdonald-Book-1995,
	author = {Macdonald, Ian G.},
	edition = {Second},
	isbn = {0-19-853489-2},
	mrclass = {05E05 (05-02 20C30 20C33 20K01 33C80 33D80)},
	mrnumber = {1354144},
	mrreviewer = {John R. Stembridge},
	note = {With contributions by A. Zelevinsky, Oxford Science Publications},
	pages = {x+475},
	publisher = {The Clarendon Press, Oxford University Press, New York},
	series = {Oxford Mathematical Monographs},
	title = {Symmetric functions and {H}all polynomials},
	year = {1995}}

@misc{MorierGenoud,
      title={On $q$-deformed {M}arkov numbers. {C}ohn matrices and perfect matchings with weighted edges}, 
      author={Sam Evans and Perrine Jouteur and Sophie Morier-Genoud and Valentin Ovsienko},
      year={2025},
      eprint={2507.19080},
      archivePrefix={arXiv},
      primaryClass={math.CO},
      url={https://arxiv.org/abs/2507.19080}, 
}

@article{MurtyUnimodal,
author = {Murty, M. Ram},
title = {Unimodal sequences: From {I}saac {N}ewton to {J}une {H}uh},
journal = {International Journal of Number Theory},
volume = {20},
number = {10},
pages = {2453-2476},
year = {2024},
doi = {10.1142/S1793042124300011},
URL = { https://doi.org/10.1142/S1793042124300011},
eprint = { https://doi.org/10.1142/S1793042124300011}
}

@article{NOVELLI2020102019,
title = {Hopf algebras of $m$-permutations, $(m+1)$-ary trees, and $m$-parking functions},
journal = {Advances in Applied Mathematics},
volume = {117},
pages = {102019},
year = {2020},
issn = {0196-8858},
doi = {https://doi.org/10.1016/j.aam.2020.102019},
url = {https://www.sciencedirect.com/science/article/pii/S0196885820300221},
author = {Jean-Christophe Novelli and Jean-Yves Thibon},}

@article{Patrias,
	author = {Patrias, Rebecca and van Willigenburg, Stephanie},
	doi = {10.4310/JOC.2020.v11.n3.a3},
	issn = {1439-6912},
	journal = {Journal of Combinatorics},
	language = {en},
	month = may,
	number = {3},
	pages = {475--493},
	title = {The probability of positivity in symmetric and quasisymmetric functions},
	url = {https://dx.doi.org/10.4310/JOC.2020.v11.n3.a3},
	volume = {11},
	year = {2020}}

@book{Petersen2015,
  title = {Eulerian Numbers},
  ISBN = {9781493930913},
  ISSN = {2296-4894},
  url = {http://dx.doi.org/10.1007/978-1-4939-3091-3},
  DOI = {10.1007/978-1-4939-3091-3},
  journal = {Birkh\"{a}user Advanced Texts Basler Lehrb\"{u}cher},
  publisher = {Springer New York},
  author = {Petersen,  T. Kyle},
  year = {2015}
}

@misc{postnikov2007,
      title={Faces of Generalized Permutohedra}, 
      author={Alexander Postnikov and Victor Reiner and Lauren Williams},
      year={2007},
      eprint={math/0609184},
      archivePrefix={arXiv},
      primaryClass={math.CO},
      url={https://arxiv.org/abs/math/0609184}, 
}

@article{Proctor1984,
  title = {Bruhat Lattices,  Plane Partition Generating Functions,  and Minuscule Representations},
  volume = {5},
  ISSN = {0195-6698},
  url = {http://dx.doi.org/10.1016/S0195-6698(84)80037-2},
  DOI = {10.1016/s0195-6698(84)80037-2},
  number = {4},
  journal = {European Journal of Combinatorics},
  publisher = {Elsevier BV},
  author = {Proctor,  Robert A.},
  year = {1984},
  month = Dec,
  pages = {331–350}
}

@article{Ratliff2004,
  title = {Note on Cyclotomic Polynomials and Prime Ideals},
  volume = {32},
  ISSN = {1532-4125},
  url = {http://dx.doi.org/10.1081/AGB-120027870},
  DOI = {10.1081/agb-120027870},
  number = {1},
  journal = {Communications in Algebra},
  publisher = {Informa UK Limited},
  author = {Ratliff,  Louis J. and Rush,  David E. and Shah,  Kishor},
  year = {2004},
  month = Mar,
  pages = {333–343}
}

@article{Remmel-Wilson-2015,
	author = {Remmel, Jeffrey B. and Wilson, Andrew T.},
	doi = {10.1016/j.jcta.2015.03.012},
	fjournal = {Journal of Combinatorial Theory. Series A},
	issn = {0097-3165},
	journal = {J. Combin. Theory Ser. A},
	mrclass = {05A18 (05A15)},
	mrnumber = {3345306},
	mrreviewer = {Damir Yeliussizov},
	pages = {242--277},
	title = {An extension of {M}ac{M}ahon's equidistribution theorem to ordered set partitions},
	url = {https://doi.org/10.1016/j.jcta.2015.03.012},
	volume = {134},
	year = {2015}}

@article{Sagan2010,
  title = {Combinatorial Interpretations of Binomial Coefficient Analogues Related to Lucas Sequences},
  volume = {10},
  ISSN = {1867-0652},
  url = {http://dx.doi.org/10.1515/integ.2010.052},
  DOI = {10.1515/integ.2010.052},
  number = {6},
  journal = {Integers},
  publisher = {Walter de Gruyter GmbH},
  author = {Sagan,  Bruce E. and Savage,  Carla D.},
  year = {2010},
  month = Jan 
}

@article {ShareshianWachs,
    AUTHOR = {Shareshian, John and Wachs, Michelle L.},
     TITLE = {Gamma-positivity of variations of {E}ulerian polynomials},
   JOURNAL = {J. Comb.},
  FJOURNAL = {Journal of Combinatorics},
    VOLUME = {11},
      YEAR = {2020},
    NUMBER = {1},
     PAGES = {1--33},
      ISSN = {2156-3527,2150-959X},
   MRCLASS = {05E05 (11B65 11B68 52B20)},
  MRNUMBER = {4015851},
MRREVIEWER = {Ivica\ Martinjak},
       DOI = {10.4310/JOC.2020.v11.n1.a1},
       URL = {https://doi-org/10.4310/JOC.2020.v11.n1.a1},
}

@article{StanleyCyclo,
  title = {Some enumerative applications of cyclotomic polynomials},
  volume = {6},
  ISSN = {2710-2335},
  url = {http://dx.doi.org/10.54550/ECA2026V6S1R4},
  DOI = {10.54550/eca2026v6s1r4},
  journal = {Enumerative Combinatorics and Applications},
  publisher = {University of Haifa},
  author = {Stanley,  Richard P.},
  year = {2025},
}

@article{sagan2020lucas-17f, 
  year     = {2020}, 
  title    = {Lucas atoms}, 
  author   = {Sagan, Bruce E. and Tirrell, Jordan}, 
  journal  = {Advances in Mathematics}, 
  issn     = {0001-8708}, 
  doi      = {10.1016/j.aim.2020.107387}, 
    pages    = {107387}, 
  volume   = {374}
}

@article{VESSENES2004579,
title = {Generalized {F}oulkes' Conjecture and tableaux construction},
journal = {Journal of Algebra},
volume = {277},
number = {2},
pages = {579-614},
year = {2004},
issn = {0021-8693},
doi = {https://doi.org/10.1016/j.jalgebra.2004.01.017},
url = {https://www.sciencedirect.com/science/article/pii/S0021869304001073},
author = {Rebecca Vessenes},
}

@article{Wachs1991,
  title = {$p,q$-{S}tirling numbers and set partition statistics},
  volume = {56},
  ISSN = {0097-3165},
  url = {http://dx.doi.org/10.1016/0097-3165(91)90020-H},
  DOI = {10.1016/0097-3165(91)90020-h},
  number = {1},
  journal = {Journal of Combinatorial Theory,  Series A},
  publisher = {Elsevier BV},
  author = {Wachs,  Michelle and White,  Dennis},
  year = {1991},
  month = Jan,
  pages = {27–46}
}

@article{Warnaar2010,
  title = {A $q$-rious positivity},
  volume = {81},
  ISSN = {1420-8903},
  url = {http://dx.doi.org/10.1007/s00010-010-0055-9},
  DOI = {10.1007/s00010-010-0055-9},
  number = {1-2},
  journal = {Aequationes mathematicae},
  publisher = {Springer Science and Business Media LLC},
  author = {Warnaar,  S. Ole and Zudilin,  W.},
  year = {2010},
  month = Oct,
  pages = {177–183}
}

@article{Zanello2018,
  title = {On {B}ergeron's Positivity Problem for $q$-Binomial Coefficients},
  volume = {25},
  ISSN = {1077-8926},
  url = {http://dx.doi.org/10.37236/7358},
  DOI = {10.37236/7358},
  number = {2},
  journal = {The Electronic Journal of Combinatorics},
  publisher = {The Electronic Journal of Combinatorics},
  author = {Zanello,  Fabrizio},
  year = {2018},
  month = apr 
}

\end{document}